\documentclass[reqno]{amsart}
\usepackage{amssymb,latexsym}
\usepackage{graphicx}
\usepackage{fancyhdr}
\numberwithin{equation}{section}
\newtheorem{thm}{Theorem}[section]
\newtheorem{theorem}[thm]{Theorem}

\newtheorem{lemma}[thm]{Lemma}
\newtheorem{definition}[thm]{Definition}
\newtheorem{corollary}[thm]{Corollary}

\begin{document}

\setcounter{page}{1}

\title[Reciprocity formulas]{Reciprocity formulas for generalized Dedekind character sums}
\author{Yuan He}
\address{School of Mathematics and Big Data, Neijiang Normal University, Neijiang 641100, Sichuan, People's Republic of China}
\email{hyyhe@aliyun.com}

\begin{abstract}
In this paper, we apply the Fourier series technique to establish a new formula for the product of the generalized Bernoulli function and the Bernoulli function. As applications, we derive some reciprocity formulas for generalized Dedekind character sums. In particular, we show that the reciprocity formulas of Berndt, Can, and others follow as special cases.
\end{abstract}

\maketitle

\section{Introduction}\label{sec1}

Let $\mathbb{N}$ be the set of positive integers, $\mathbb{N}_{0}$ the set of non-negative integers, $\mathbb{Z}$ the set of integers, $\mathbb{R}$ the set of real numbers, and $\mathbb{C}$ the set of complex numbers. Denote by $(a,b)$ the greatest common divisor of $a,b\in\mathbb{Z}$; by $\{x\}$ the fractional part of $x\in\mathbb{R}$; by $[x]$ the integer part of $x\in\mathbb{R}$, satisfying $[x]=x-\{x\}$; by $\delta_{l,k}$ the Kronecker delta-function, defined for $l,k\in\mathbb{N}_{0}$ by $\delta_{l,k}=1$ if $l=k$ and $0$ otherwise; and by $\delta_{\mathbb{Z}}(x)$ the indicator function of
$\mathbb{Z}$, defined by $\delta_{\mathbb{Z}}(x)=1$ if $x\in\mathbb{Z}$ and $0$ otherwise. Let $\overline{B}_{n}(x)$ denote the $n$-th Bernoulli function, defined for $n\in\mathbb{N}_{0}$ and $x\in\mathbb{R}$ by
\begin{equation*}
\overline{B}_{n}(x)=B_{n}(\{x\})+\frac{1}{2}\delta_{1,n}\delta_{\mathbb{Z}}(x),
\end{equation*}
where $B_{n}(x)$ is the Bernoulli polynomial of degree $n$, defined by the generating function
\begin{equation*}
\frac{te^{xt}}{e^{t}-1}=\sum_{n=0}^{\infty}B_{n}(x)\frac{t^{n}}{n!}\quad(|t|<2\pi).
\end{equation*}
For $a\in\mathbb{Z}$ and $b\in\mathbb{N}$, the classical Dedekind sums $s(a,b)$ are defined by
\begin{equation}\label{eq1.1}
s(a,b)=\sum_{r=1}^{b}\overline{B}_{1}\biggl(\frac{ar}{b}\biggl)\overline{B}_{1}\biggl(\frac{r}{b}\biggl).
\end{equation}
In the 1880s, Dedekind \cite{dedekind} introduced \eqref{eq1.1} in connection with the transformation formulas of the Dedekind eta-function. From this, he deduced his reciprocity theorem: if
$a,b\in\mathbb{N}$ with $(a,b)=1$, then
\begin{equation}\label{eq1.2}
s(a,b)+s(b,a)=-\frac{1}{4}+\frac{1}{12}\biggl(\frac{a}{b}+\frac{b}{a}+\frac{1}{ab}\biggl).
\end{equation}
Since the early 20th century, many mathematicians have studied these sums and their generalizations, and have established the corresponding reciprocity formulas; see, e.g., \cite{beck1,beck2,berndt4,nagasaka} and the references therein.

This paper is concerned with Dedekind character sums introduced by Berndt \cite{berndt1,berndt3}, together with their generalizations. In 1973, while studying the transformation formulas of a generalized Eisenstein series, Berndt \cite{berndt1} was led to define the sums $s(a,b;\chi,\overline{\chi})$ by
\begin{equation}\label{eq1.3}
s(a,b;\chi,\overline{\chi})=\sum_{r=1}^{bq}\chi(r)\overline{B}_{1}\biggl(\frac{ar}{b},\overline{\chi}\biggl)\overline{B}_{1}\biggl(\frac{r}{bq}\biggl),
\end{equation}
where $a,b\in\mathbb{Z}$ with $b\geq1$, $\chi$ is a Dirichlet character modulo $q$ ($q\in\mathbb{N}$), $\overline{\chi}$ denotes the complex conjugation of $\chi$, and $\overline{B}_{n}(x,\chi)$ is the $n$-th generalized Bernoulli function, defined for $n\in\mathbb{N}_{0}$ and $x\in\mathbb{R}$ by
\begin{equation*}
\overline{B}_{n}(x,\chi)=q^{n-1}\sum_{r=1}^{q}\chi(r)\overline{B}_{n}\biggl(\frac{r+x}{q}\biggl).
\end{equation*}
Meanwhile, Berndt \cite[Theorem 4]{berndt1} used his transformation formulas to derive the following reciprocity theorem. Let $q\in\mathbb{N}$ with $q\geq2$, and let $\chi$ be a non-principal primitive character modulo $q$. Then, for $a,b\in\mathbb{N}$ with $(a,b)=1$ and either $a\equiv0$ or $b\equiv0\pmod{q}$,
\begin{equation}\label{eq1.4}
s(a,b;\chi,\overline{\chi})+s(b,a;\overline{\chi},\chi)=B_{1}(\chi)B_{1}(\overline{\chi}),
\end{equation}
where $B_{n}(\chi)=\overline{B}_{n}(0,\chi)$. Later, in 1975, Berndt \cite{berndt3} introduced the sums $s_{1}(a,b;\chi)$ and $s_{2}(a,b;\chi)$, defined for a Dirichlet character $\chi$ modulo $q$ ($q\in\mathbb{N}$), and for $a,b\in\mathbb{Z}$ with $b\geq1$, by
\begin{equation}\label{eq1.5}
s_{1}(a,b;\chi)=\sum_{r=1}^{bq}\chi(r)\overline{B}_{1}\biggl(\frac{ar}{b}\biggl)\overline{B}_{1}\biggl(\frac{r}{bq}\biggl),
\end{equation}
and
\begin{equation}\label{eq1.6}
s_{2}(a,b;\chi)=\sum_{r=1}^{bq}\overline{B}_{1}\biggl(\frac{ar}{b},\chi\biggl)\overline{B}_{1}\biggl(\frac{r}{bq}\biggl).
\end{equation}
Moreover, by applying his transformation formulas for another generalized Eisenstein series, Berndt \cite[Theorem 7]{berndt3} showed that for $a,b,q\in\mathbb{N}$ with $(a,b)=1$ and $q\geq2$, and for an even, non-principal primitive character $\chi$ modulo $q$, if $(b,q)=1$, then
\begin{equation}\label{eq1.7}
s_{1}(a,b;\chi)+s_{2}(b,a;\chi)=\frac{a}{2b}B_{2}(\chi);
\end{equation}
and if $b\equiv0\pmod{q}$, then
\begin{equation}\label{eq1.8}
s_{1}(a,b;\chi)+s_{2}(b,a;\chi)=\frac{a}{2b}B_{2}(\chi)+\frac{\overline{\chi}(a)}{2ab}B_{2}(\chi).
\end{equation}
In 1977, Berndt \cite{berndt4} provided an alternative proof of \eqref{eq1.4}, \eqref{eq1.7} and \eqref{eq1.8} by applying the (periodic) Poisson summation formula.
In 2007, Cenkci, Can and Kurt \cite{cenkci} studied some generalizations of \eqref{eq1.3} and found the corresponding reciprocity formulas.
In 2017, Da\u{g}l{\i} and Can \cite{dagli} explored similar generalizations by applying Berndt's character analogue of the Euler--Maclaurin summation formula \cite{berndt2}.
Furthermore, in 2020, Can \cite{can} investigated the sums $s_{m,n}(a,b,c;\chi_{1},\chi_{2};x,y,z)$ and $s_{m,n}^{*}(a,b,c;\chi_{1},\chi_{2};x,y,z)$, defined for Dirichlet characters $\chi_{1}$ and $\chi_{2}$ modulo $q$ ($q\in\mathbb{N}$), and for $m,n\in\mathbb{N}_{0}$, $a,b,c\in\mathbb{Z}$ with $c\geq1$, and $x,y,z\in\mathbb{R}$, by
\begin{eqnarray}\label{eq1.9}
&&s_{m,n}(a,b,c;\chi_{1},\chi_{2};x,y,z)\nonumber\\
&&\qquad=\sum_{r=1}^{cq}\chi_{1}(r)\overline{B}_{m}\biggl(\frac{a(r+z)}{c}+x,\chi_{2}\biggl)\overline{B}_{n}\biggl(\frac{b(r+z)}{cq}+y\biggl),
\end{eqnarray}
and
\begin{eqnarray}\label{eq1.10}
&&s_{m,n}^{*}(a,b,c;\chi_{1},\chi_{2};x,y,z)\nonumber\\
&&\qquad=\sum_{r=1}^{c}\overline{B}_{m}\biggl(\frac{aq(r+z)}{c}+x,\chi_{1}\biggl)\overline{B}_{n}\biggl(\frac{bq(r+z)}{c}+y,\chi_{2}\biggl).
\end{eqnarray}
In particular, Can \cite[Theorem 2.2]{can} established a reciprocity formula relating the sums \eqref{eq1.9} and \eqref{eq1.10}, which holds for $m,n,a,b,c,q\in\mathbb{N}$ satisfying $(a,b)=(b,c)=(c,a)=1$ and $q\geq2$, and for two non-principal primitive characters $\chi_{1}$ and $\chi_{2}$ modulo $q$.

Naturally, one may ask whether the hypothesis $(a,b)=(b,c)=(c,a)=1$ in Can's reciprocity formula \cite[Theorem 2.2]{can} can be removed. One may also ask whether Berndt's reciprocity formulas \eqref{eq1.7} and \eqref{eq1.8} can be generalized in the same way as \eqref{eq1.4}.

To address these two questions, we study the aforementioned sums from a different perspective. Specifically, we apply the Fourier series technique to establish the following formula for the product of the generalized Bernoulli function and the Bernoulli function, which plays a key role in the sequel.

\begin{theorem}\label{thm1.1} Let $m,n,q\in\mathbb{N}$, and $a,b\in\mathbb{Z}$ with $a\not=0$ and $b\not=0$. Let $\chi$ be a primitive Dirichlet character modulo $q$. Then, for $x,y,z\in\mathbb{R}$,
\begin{eqnarray}\label{eq1.11}
&&\overline{B}_{m}(ax+y,\overline{\chi})\overline{B}_{n}\biggl(\frac{bx}{q}+z\biggl)\nonumber\\
&&\qquad=nq^{m-1}b^{n-1}\chi(-1)\mathrm{sgn}(b)\sum_{j=1}^{m}\binom{m}{j}\frac{(-1)^{j}a^{m-j}}{m+n-j}\sum_{r=1}^{q}\overline{\chi}(r)\nonumber\\
&&\qquad\qquad\times\sum_{l=1}^{|b|}\overline{B}_{j}\biggl(\frac{a(l+z)}{b}-\frac{y-r}{q}\biggl)\overline{B}_{m+n-j}\biggl(\frac{x}{q}+\frac{l+z}{b}\biggl)\nonumber\\
&&\qquad\quad+m(aq)^{m-1}\chi(-1)\mathrm{sgn}(a)\sum_{j=1}^{n}\binom{n}{j}\frac{(-1)^{j}b^{n-j}}{m+n-j}\sum_{r=1}^{q}\overline{\chi}(r)\nonumber\\
&&\qquad\qquad\times\sum_{l=1}^{|a|}\overline{B}_{j}\biggl(\frac{bl}{a}+\frac{by-br}{aq}-z\biggl)\overline{B}_{m+n-j}\biggl(\frac{x}{q}+\frac{l}{a}+\frac{y-r}{aq}\biggl)\nonumber\\
&&\qquad\quad+\frac{mb^{n}\overline{B}_{m+n}(ax+y,\overline{\chi})}{(aq)^{n}(m+n)}+\frac{nq^{m-1}a^{m}\chi(0)G(1,\overline{\chi})\overline{B}_{m+n}(\frac{bx}{q}+z)}{b^{m}(m+n)}\nonumber\\
&&\qquad\quad+\frac{(-1)^{n-1}m!n!(a,b)^{m+n}\chi(\frac{b}{(a,b)})\overline{B}_{m+n}\bigl(\frac{by-aqz}{(a,b)},\overline{\chi}\bigl)}{(aq)^{n}b^{m}(m+n)!}\nonumber\\
&&\qquad\quad-\delta_{1,m}\delta_{1,n}\mathrm{sgn}(ab)\frac{\delta_{\mathbb{Z}}(ax+y)\delta_{\mathbb{Z}}(\frac{bx}{q}+z)\overline{\chi}(-ax-y)}{4}.
\end{eqnarray}
\end{theorem}

Clearly, the case $q=1$ in Theorem \ref{thm1.1} is consistent with the result in \cite[Theorem 1.1]{he1}. Moreover, it can be used to establish a reciprocity formula for the generalized Dedekind--Rademacher sums introduced by Hall, Wilson and Zagier \cite{hall}; see \cite[Theorem 1.2]{he1} for details.

As another application of Theorem \ref{thm1.1}, we are able to remove the hypothesis $(a,b)=(b,c)=(c,a)=1$ in the reciprocity formula of Can \cite[Theorem 2.2]{can}. Our result is stated as follows.

\begin{theorem}\label{thm1.2} Let $m,n,a,b,c,q\in\mathbb{N}$ with $q\geq 2$. Let $\chi_{1}$ and $\chi_{2}$ be non-principal primitive Dirichlet characters modulo $q$. Then, for $x,y,z\in\mathbb{R}$,
\begin{eqnarray}\label{eq1.12}
&&s_{m,n}(a,b,c;\chi_{1},\overline{\chi_{2}};x,y,z)\nonumber\\
&&\qquad=n\biggl(\frac{b}{cq}\biggl)^{n-1}\sum_{j=1}^{m}\binom{m}{j}\frac{(\frac{a}{c})^{m-j}}{m+n-j}s_{j,m+n-j}^{*}(-a,c,b;\overline{\chi_{2}},\chi_{1};x,z,y)\nonumber\\
&&\qquad\quad+m\biggl(\frac{a}{c}\biggl)^{m-1}\sum_{j=0}^{n}\binom{n}{j}\frac{(-1)^{j}(\frac{b}{cq})^{n-j}}{m+n-j}s_{m+n-j,j}(c,b,a;\overline{\chi_{2}},\chi_{1};z,-y,x)\nonumber\\
&&\qquad\quad-\frac{\delta_{1,m}\delta_{1,n}\chi_{2}(-1)}{4}\sum_{k=1}^{cq}\chi_{1}(k)\delta_{\mathbb{Z}}\biggl(\frac{a(k+z)}{c}+x\biggl)\delta_{\mathbb{Z}}\biggl(\frac{b(k+z)}{cq}+y\biggl)\nonumber\\
&&\qquad\qquad\times\overline{\chi_{2}}\biggl(\frac{a(k+z)}{c}+x\biggl).
\end{eqnarray}
\end{theorem}

In order to generalize Berndt's reciprocity formulas \eqref{eq1.7} and \eqref{eq1.8} in the same spirit as Theorem \ref{thm1.2}, we introduce the following three sums:
\begin{equation*}
s_{m,n}^{(1)}(a,b,c;\chi;x,y,z),\quad s_{m,n}^{(2)}(a,b,c;\chi;x,y,z),\quad s_{m,n}^{(3)}(a,b,c;\chi;x,y,z),
\end{equation*}
defined for a Dirichlet character $\chi$ modulo $q$ ($q\in\mathbb{N}$), and for $m,n\in\mathbb{N}_{0}$, $a,b,c\in\mathbb{Z}$ with $c\geq1$, and $x,y,z\in\mathbb{R}$, by
\begin{equation}\label{eq1.13}
s_{m,n}^{(1)}(a,b,c;\chi;x,y,z)=\sum_{r=1}^{cq}\chi(r)\overline{B}_{m}\biggl(\frac{a(r+z)}{c}+x\biggl)\overline{B}_{n}\biggl(\frac{b(r+z)}{cq}+y\biggl),
\end{equation}
\begin{equation}\label{eq1.14}
s_{m,n}^{(2)}(a,b,c;\chi;x,y,z)=\sum_{r=1}^{cq}\overline{B}_{m}\biggl(\frac{a(r+z)}{c}+x,\chi\biggl)\overline{B}_{n}\biggl(\frac{b(r+z)}{cq}+y\biggl),
\end{equation}
and
\begin{equation}\label{eq1.15}
s_{m,n}^{(3)}(a,b,c;\chi;x,y,z)=\sum_{r=1}^{c}\overline{B}_{m}\biggl(\frac{aq(r+z)}{c}+x,\chi\biggl)\overline{B}_{n}\biggl(\frac{bq(r+z)}{c}+y\biggl).
\end{equation}
Applying Theorem \ref{thm1.1}, we obtain a reciprocity formula relating the sums \eqref{eq1.13}, \eqref{eq1.14} and \eqref{eq1.15}. This formula recovers \eqref{eq1.7} and \eqref{eq1.8}, as well as another reciprocity formulas of Berndt in \cite[Corollary 8.5]{berndt4}, as special cases. The precise statement reads as follows.

\begin{theorem}\label{thm1.3} Let $m,n,a,b,c,q\in\mathbb{N}$ with $q\geq 2$. Let $\chi$ be a non-principal primitive Dirichlet character modulo $q$. Then, for $x,y,z\in\mathbb{R}$,
\begin{eqnarray}\label{eq1.16}
&&s_{m,n}^{(2)}(a,b,c;\chi;x,y,z)\nonumber\\
&&\qquad=n\biggl(\frac{b}{cq}\biggl)^{n-1}\sum_{j=1}^{m}\binom{m}{j}\frac{(\frac{a}{c})^{m-j}}{m+n-j}s_{j,m+n-j}^{(3)}(-a,c,b;\chi;x,z,y)\nonumber\\
&&\qquad\quad+m\biggl(\frac{a}{c}\biggl)^{m-1}\sum_{j=0}^{n}\binom{n}{j}\frac{(-1)^{j}(\frac{b}{cq})^{n-j}}{m+n-j}s_{m+n-j,j}^{(1)}(c,b,a;\chi;z,-y,x)\nonumber\\
&&\qquad\quad+\frac{(-1)^{n-1}m!n!(a,b)^{m+n}c\overline{\chi}(\frac{b}{(a,b)})\overline{B}_{m+n}\bigl(\frac{bx-aqy}{(a,b)},\chi\bigl)}{a^{n}q^{n-1}b^{m}(m+n)!}\nonumber\\
&&\qquad\quad-\frac{\delta_{1,m}\delta_{1,n}\chi(-1)}{4}\sum_{k=1}^{cq}\delta_{\mathbb{Z}}\biggl(\frac{a(k+z)}{c}+x\biggl)\delta_{\mathbb{Z}}\biggl(\frac{b(k+z)}{cq}+y\biggl)\nonumber\\
&&\qquad\qquad\times\chi\biggl(\frac{a(k+z)}{c}+x\biggl).
\end{eqnarray}
\end{theorem}

Finally, as a natural complement to Theorem \ref{thm1.2}, we introduce two more sums
$\widehat{s}_{m,n}(a,b,c;\chi_{1},\chi_{2};x,y,z)$ and $\widetilde{s}_{m,n}(a,b,c;\chi_{1},\chi_{2};x,y,z)$,
defined for Dirichlet characters $\chi_{1}$ modulo $q_{1}$ and $\chi_{2}$ modulo $q_{2}$ ($q_{1},q_{2}\in\mathbb{N}$), and for $m,n\in\mathbb{N}_{0}$, $a,b,c\in\mathbb{Z}$ with $c\geq1$, and $x,y,z\in\mathbb{R}$, by
\begin{eqnarray}\label{eq1.17}
&&\widehat{s}_{m,n}(a,b,c;\chi_{1},\chi_{2};x,y,z)\nonumber\\
&&\qquad=\sum_{r=1}^{cq_{1}q_{2}}\chi_{1}(r)\overline{B}_{m}\biggl(\frac{a(r+z)}{c}+x,\chi_{2}\biggl)\overline{B}_{n}\biggl(\frac{b(r+z)}{cq_{1}q_{2}}+y\biggl),
\end{eqnarray}
and
\begin{eqnarray}\label{eq1.18}
&&\widetilde{s}_{m,n}(a,b,c;\chi_{1},\chi_{2};x,y,z)\nonumber\\
&&\qquad=\sum_{r=1}^{c}\overline{B}_{m}\biggl(\frac{aq_{1}q_{2}(r+z)}{c}+x,\chi_{1}\biggl)\overline{B}_{n}\biggl(\frac{bq_{1}q_{2}(r+z)}{c}+y,\chi_{2}\biggl),
\end{eqnarray}
Just as the sums \eqref{eq1.9} and \eqref{eq1.10} do, the sums \eqref{eq1.17} and \eqref{eq1.18} satisfy a reciprocity formula of the same type, stated in the following theorem.

\begin{theorem}\label{thm1.4}Let $m,n,a,b,c,q_{1},q_{2}\in\mathbb{N}$ with $q_{1}\geq 2$ and $q_{2}\geq2$. Let $\chi_{1}$ and $\chi_{2}$ be non-principal primitive Dirichlet characters modulo $q_{1}$ and $q_{2}$, respectively. Then, for $x,y,z\in\mathbb{R}$,
\begin{eqnarray}\label{eq1.19}
&&\widehat{s}_{m,n}(a,b,c;\chi_{1},\overline{\chi_{2}};x,y,z)\nonumber\\
&&\qquad=n\biggl(\frac{b}{cq_{1}q_{2}}\biggl)^{n-1}\sum_{j=1}^{m}\binom{m}{j}\frac{(\frac{a}{c})^{m-j}}{m+n-j}
\widetilde{s}_{j,m+n-j}(-a,c,b;\overline{\chi_{2}},\chi_{1};x,z,y)\nonumber\\
&&\qquad\quad+m\biggl(\frac{a}{c}\biggl)^{m-1}\sum_{j=0}^{n}\binom{n}{j}\frac{(-1)^{j}(\frac{b}{cq_{1}q_{2}})^{n-j}}{m+n-j}\widehat{s}_{m+n-j,j}(c,b,a;\overline{\chi_{2}},\chi_{1};z,-y,x)\nonumber\\
&&\qquad\quad-\frac{\delta_{1,m}\delta_{1,n}\chi_{2}(-1)}{4}\sum_{k=1}^{cq_{1}q_{2}}\chi_{1}(k)\delta_{\mathbb{Z}}\biggl(\frac{a(k+z)}{c}+x\biggl)
\delta_{\mathbb{Z}}\biggl(\frac{b(k+z)}{cq_{1}q_{2}}+y\biggl)\nonumber\\
&&\qquad\qquad\times\overline{\chi_{2}}\biggl(\frac{a(k+z)}{c}+x\biggl).
\end{eqnarray}
\end{theorem}

This paper is organized as follows. In Section \ref{sec2}, we give some preliminary results. Section \ref{sec3} is devoted to the detailed proofs of Theorems \ref{thm1.1}, \ref{thm1.2}, \ref{thm1.3} and \ref{thm1.4}. In Section \ref{sec4}, we discuss some special cases of Theorems \ref{thm1.2}, \ref{thm1.3} and \ref{thm1.4}. We show that Berndt's reciprocity formulas \eqref{eq1.4}, \eqref{eq1.7} and \eqref{eq1.8}, as well as the reciprocity formulas in \cite{berndt4,cenkci,dagli}, hold as special cases.

\section{Some preliminary results}\label{sec2}

We first recall the following definitions of the generalized Apostol-Bernoulli polynomials and the generalized Apostol-Bernoulli functions, given in \cite[Definitions 2.1 and 2.7]{he2}.

\begin{definition}\label{def2.1} Let $n\in\mathbb{N}_{0}$, $q\in\mathbb{N}$, and $\lambda\in\mathbb{C}\setminus\{0\}$. For a Dirichlet character $\chi$ modulo $q$, the generalized Apostol-Bernoulli polynomials $\beta_{n,\chi}(x,\lambda)$ are defined by the generating function
\begin{equation}\label{eq2.1}
\sum_{r=1}^{q}\frac{\chi(r)\lambda^{r}te^{(r+x)t}}{\lambda^{q}e^{qt}-1}=\sum_{n=0}^{\infty}\beta_{n,\chi}(x,\lambda)\frac{t^{n}}{n!}\quad(|t+\log\lambda|<2\pi/q),
\end{equation}
where $\lambda^{z}$ is the general power function defined by $\lambda^{z}=e^{z\log\lambda}$ for $z\in\mathbb{C}$ with the principal branch of the logarithm satisfying
\begin{equation*}
\log \lambda=\log|\lambda|+\mathrm{i}\arg \lambda \quad(-\pi<\arg \lambda\leqslant\pi).
\end{equation*}
In particular, $\beta_{n,\chi}(\lambda)=\beta_{n,\chi}(0,\lambda)$ are called the generalized Apostol-Bernoulli numbers.
\end{definition}

\begin{definition}\label{def2.2} Let $n\in\mathbb{N}_{0}$, $q\in\mathbb{N}$, $x\in\mathbb{R}$, and $\lambda\in\mathbb{C}\setminus\{0\}$. For a Dirichlet character $\chi$ modulo $q$, the generalized Apostol-Bernoulli functions $\overline{\beta}_{n,\chi}(x,\lambda)$ are defined by
\begin{equation}\label{eq2.2}
\overline{\beta}_{n,\chi}(x,\lambda)=
\beta_{n,\chi}(\{x\}_{\chi},\lambda)+\frac{1}{2}\delta_{1,n}\delta_{\mathbb{Z}}(x)\lambda^{-x}\chi(-x),
\end{equation}
where
\begin{equation*}
\beta_{n,\chi}(\{x\}_{\chi},\lambda)=q^{n-1}\sum_{k=1}^{q}\chi(k)\lambda^{k-q[\frac{x+k}{q}]}\beta_{n}\biggl(\biggl\{\frac{x+k}{q}\biggl\},\lambda^{q}\biggl),
\end{equation*}
and $\beta_{n}(x,\lambda)$ are the Apostol-Bernoulli polynomials (for $q=1$ in Definition \ref{def2.1}, they satisfy $\lambda\beta_{n}(x,\lambda)=\beta_{n,\chi}(x-1,\lambda))$.
\end{definition}

It should be noted that Definition \ref{def2.1} includes information on the Bernoulli, Euler, Apostol-Bernoulli, Apostol-Euler, generalized Bernoulli, and generalized Euler polynomials, while Definition \ref{def2.2} includes information on the corresponding functions. It is also shown in \cite{he2} that the generalized Apostol-Bernoulli functions have the following Fourier series.

\begin{lemma}\label{lem2.3} Let $n,q\in\mathbb{N}$, $x\in\mathbb{R}$, and $\lambda\in\mathbb{C}\setminus\{0\}$. Let $\chi$ be a Dirichlet character modulo $q$. Then
\begin{equation}\label{eq2.3}
\overline{\beta}_{n,\chi}(x,\lambda)=-\frac{q^{n-1}n!}{\lambda^{x}(2\pi \mathrm{i})^{n}}\sideset{}{'}\sum_{k=-\infty}^{+\infty}\frac{G(k,\chi)e^{\frac{2\pi\mathrm{i}kx}{q}}}{(k-\frac{q\log\lambda}{2\pi\mathrm{i}})^{n}},
\end{equation}
where the dash denotes throughout that undefined terms are excluded from the
sum, and $G(k,\chi)$ is the Gauss sum defined by
\begin{equation*}
G(k,\chi)=\sum_{r=1}^{q}\chi(r)e^{\frac{2\pi \mathrm{i}rk}{q}}\quad(k\in\mathbb{Z}).
\end{equation*}
\end{lemma}

\begin{proof}
See \cite[Theorem 3.1]{he2} for details.
\end{proof}

Clearly, for a primitive character $\chi$ modulo $q$ ($q\in\mathbb{N}$), the following identity holds (see, e.g., \cite[Theorem 8.15]{apostol}):
\begin{equation}\label{eq2.4}
G(k,\chi)=\overline{\chi}(k)G(1,\chi)\quad(k\in\mathbb{Z}).
\end{equation}
Hence, taking $\lambda=1$ in Lemma \ref{lem2.3} and using \eqref{eq2.4}, we obtain that, for a primitive Dirichlet character $\chi$ modulo $q$ ($q\in\mathbb{N}$), and for $n\in\mathbb{N}$ and $x\in\mathbb{R}$,
\begin{equation}\label{eq2.5}
\overline{B}_{n}(x,\overline{\chi})=-\frac{q^{n-1}n!G(1,\overline{\chi})}{(2\pi \mathrm{i})^{n}}\sideset{}{'}\sum_{k=-\infty}^{+\infty}\frac{\chi(k)e^{\frac{2\pi\mathrm{i}kx}{q}}}{k^{n}}.
\end{equation}
In particular, the case $q=1$ in \eqref{eq2.5} gives the Fourier series of the Bernoulli functions:
\begin{equation}\label{eq2.6}
\overline{B}_{n}(x)=-\frac{n!}{(2\pi \mathrm{i})^{n}}\sideset{}{'}\sum_{k=-\infty}^{+\infty}\frac{e^{2\pi\mathrm{i}kx}}{k^{n}}\quad(n\in\mathbb{N},x\in\mathbb{R}).
\end{equation}
It is worth noting that \eqref{eq2.6} implies
\begin{equation}\label{eq2.7}
\overline{B}_{n}(-x)=(-1)^{n}\overline{B}_{n}(x)\quad(n\in\mathbb{N},x\in\mathbb{R}).
\end{equation}
Moreover, for a Dirichlet character $\chi$ modulo $q$ ($q\in\mathbb{N}$), and for $n\in\mathbb{N}$ and $x\in\mathbb{R}$, we see that
\begin{equation}\label{eq2.8}
q^{n-1}\sum_{r=1}^{q}\overline{\chi}(r)\overline{B}_{n}\biggl(\frac{x+r}{q}\biggl)=\overline{B}_{n}(x,\overline{\chi}).
\end{equation}

The next result appears in \cite[Lemma 2.7]{he1}. Here we provide an alternative proof.

\begin{lemma}\label{lem2.4} Let $j\in\mathbb{N}$, and $b,r\in\mathbb{Z}$ with $b\not=0$. Then, for $x\in\mathbb{R}$,
\begin{equation}\label{eq2.9}
\sideset{}{'}\sum_{d=-\infty}^{+\infty}\frac{e^{2\pi \mathrm{i}dx}}{\bigl(d+\frac{r}{b}\bigl)^{j}}
=\frac{(-1)^{j-1}(2\pi\mathrm{i})^{j}\mathrm{sgn}(b)}{j!b^{1-j}}\sum_{l=1}^{|b|}e^{\frac{2\pi \mathrm{i}(l-x)r}{b}}\overline{B}_{j}\biggl(\frac{l-x}{b}\biggl).
\end{equation}
\end{lemma}

\begin{proof} We only prove the case $b>0$ (the case $b<0$ can be proved in a similar manner). Write $r=bq+s$, where $q,s\in\mathbb{Z}$ with $0\leq s\leq b-1$. We have
\begin{equation*}
\sideset{}{'}\sum_{d=-\infty}^{+\infty}\frac{e^{2\pi \mathrm{i}dx}}{\bigl(d+\frac{r}{b}\bigl)^{j}}
=e^{-2\pi \mathrm{i}qx}\sideset{}{'}\sum_{d=-\infty}^{+\infty}\frac{e^{2\pi \mathrm{i}dx}}{\bigl(d+\frac{s}{b}\bigl)^{j}}.
\end{equation*}
Thus, it suffices to prove the case $0\leq r\leq b-1$ in \eqref{eq2.9}. It is easy to see that
\begin{equation}\label{eq2.10}
\arg e^{\frac{2\pi \mathrm{i}r}{b}}=\begin{cases}
\frac{2\pi r}{b},  &\text{if $0\leq r\leq b/2$},\\
\frac{2\pi r}{b}-2\pi,  &\text{if $b/2<r\leq b-1$}.
\end{cases}
\end{equation}
Setting $q=1$ and $\lambda=e^{\frac{2\pi \mathrm{i}r}{b}}$ in Lemma \ref{lem2.3}, we know from \eqref{eq2.10} that
\begin{eqnarray}\label{eq2.11}
\overline{\beta}_{j}(x,e^{\frac{2\pi \mathrm{i}r}{b}})&=&-\frac{j!e^{-\frac{2\pi \mathrm{i}rx}{b}}}{(2\pi \mathrm{i})^{j}}\sideset{}{'}\sum_{d=-\infty}^{+\infty}\frac{e^{2\pi\mathrm{i}dx}}{(d-\frac{r}{b})^{j}}\nonumber\\
&=&-\frac{(-1)^{j}j!e^{-\frac{2\pi \mathrm{i}rx}{b}}}{(2\pi \mathrm{i})^{j}}\sideset{}{'}\sum_{d=-\infty}^{+\infty}\frac{e^{-2\pi\mathrm{i}dx}}{(d+\frac{r}{b})^{j}},
\end{eqnarray}
where $\overline{\beta}_{j}(x,\lambda)$ are the Apostol-Bernoulli functions (for $q=1$ in Definition \ref{def2.2}, they satisfy $\overline{\beta}_{j}(x,\lambda)=\overline{\beta}_{j,\chi}(x,\lambda)$). On the other hand, it is shown in \cite[Corollary 3.9]{he2} that, for $a,b,j\in\mathbb{N}$, and for $\lambda,\mu\in\mathbb{C}\setminus\{0\}$ with $\lambda^{a}=\mu^{b}$,
\begin{equation}\label{eq2.12}
a^{j-1}\sum_{l=0}^{a-1}\lambda^{l}\overline{\beta}_{j}\biggl(bx+\frac{bl}{a},\mu\biggl)=b^{j-1}\sum_{l=0}^{b-1}\mu^{l}\overline{\beta}_{j}\biggl(ax+\frac{al}{b},\lambda\biggl).
\end{equation}
Taking $a=\lambda=1$ and $\mu=e^{\frac{2\pi \mathrm{i}r}{b}}$, and replacing $x$ by $x/b$ in \eqref{eq2.12}, we obtain
\begin{eqnarray}\label{eq2.13}
\overline{\beta}_{j}(x,e^{\frac{2\pi \mathrm{i}r}{b}})&=&b^{j-1}\sum_{l=0}^{b-1}e^{\frac{2\pi \mathrm{i}rl}{b}}\overline{B}_{j}\biggl(\frac{x+l}{b}\biggl)\nonumber\\
&=&b^{j-1}\sum_{l=1}^{b}e^{\frac{2\pi \mathrm{i}rl}{b}}\overline{B}_{j}\biggl(\frac{x+l}{b}\biggl).
\end{eqnarray}
Therefore, combining \eqref{eq2.11} and \eqref{eq2.13}, we conclude that the case $0\leq r\leq b-1$ in \eqref{eq2.9} holds. This completes the proof of Lemma \ref{lem2.4}.
\end{proof}

We remark that taking $r=0$ in \eqref{eq2.13} and then replacing $b$ by $a$ gives Raabe's formula \cite{raabe}: for $a,j\in\mathbb{N}$ and $x\in\mathbb{R}$,
\begin{equation}\label{eq2.14}
a^{j-1}\sum_{l=0}^{a-1}\overline{B}_{j}\biggl(\frac{x+l}{a}\biggl)=\overline{B}_{j}(x).
\end{equation}

It follows from Lemma \ref{lem2.4} that the following result holds.

\begin{lemma}\label{lem2.5} Let $j,q\in\mathbb{N}$, and $y,z\in\mathbb{R}$. Let $a,b,\overline{a},\overline{b},k\in\mathbb{Z}$ with $a\neq0$, $b\neq0$, $a\overline{a}+b\overline{b}=(a,b)$, and $(a,b)\mid k$. Let $\chi$ be a primitive Dirichlet character modulo $q$. Then
\begin{eqnarray}\label{eq2.15}
&&\sideset{}{'}\sum_{d=-\infty}^{+\infty}\frac{\chi(\frac{k\overline{a}}{(a,b)}+\frac{bd}{(a,b)})e^{\frac{2\pi\mathrm{i}d(by-aqz)}{(a,b)q}}}{\bigl(d+\frac{k\overline{a}}{b}\bigl)^{j}}\nonumber\\
&&\qquad=\frac{(-1)^{j-1}(2\pi \mathrm{i})^{j}\mathrm{sgn}(b)(a,b)^{1-j}G(1,\chi)}{qj!b^{1-j}}\sum_{r=1}^{q}\overline{\chi}(r)\nonumber\\
&&\qquad\qquad\times
\sum_{l=1}^{|\frac{b}{(a,b)}|}e^{\frac{2\pi \mathrm{i}(l-\frac{by-aqz}{(a,b)q})k\overline{a}}{b}}\overline{B}_{j}\biggl(\frac{(a,b)l-\frac{by-br-aqz}{q}}{b}\biggl),
\end{eqnarray}
and
\begin{eqnarray}\label{eq2.16}
&&\sideset{}{'}\sum_{d=-\infty}^{+\infty}\frac{\chi(\frac{k\overline{a}}{(a,b)}+\frac{bd}{(a,b)})e^{\frac{2\pi\mathrm{i}d(by-aqz)}{(a,b)q}}}{\bigl(d-\frac{k\overline{b}}{a}\bigl)^{j}}\nonumber\\
&&\qquad=\frac{(-1)^{j-1}(2\pi \mathrm{i})^{j}\mathrm{sgn}(a)(a,b)^{1-j}G(1,\chi)}{qj!a^{1-j}}\sum_{r=1}^{q}\overline{\chi}(r)e^{-\frac{2\pi\mathrm{i}kr}{aq}}\nonumber\\
&&\qquad\qquad\times
\sum_{l=1}^{|\frac{a}{(a,b)}|}e^{-\frac{2\pi \mathrm{i}(l-\frac{by-aqz}{(a,b)q})k\overline{b}}{a}}\overline{B}_{j}\biggl(\frac{(a,b)l-\frac{by-br-aqz}{q}}{a}\biggl).
\end{eqnarray}
\end{lemma}

\begin{proof}
It is well known that the finite Fourier expansion of a primitive Dirichlet character $\chi$ modulo $q$ ($q\in\mathbb{N}$) has the form (see, e.g., \cite[Theorem 8.20]{apostol}):
\begin{equation}\label{eq2.17}
\chi(m)=\frac{G(1,\chi)}{q}\sum_{r=1}^{q}\overline{\chi}(r)e^{-\frac{2\pi\mathrm{i}mr}{q}}\quad(m\in\mathbb{N}).
\end{equation}
From Lemma \ref{lem2.4} and \eqref{eq2.17}, it follows that
\begin{eqnarray*}
&&\sideset{}{'}\sum_{d=-\infty}^{+\infty}\frac{\chi(\frac{k\overline{a}}{(a,b)}+\frac{bd}{(a,b)})e^{\frac{2\pi\mathrm{i}d(by-aqz)}{(a,b)q}}}{\bigl(d+\frac{k\overline{a}}{b}\bigl)^{j}}\nonumber\\
&&\qquad=\frac{G(1,\chi)}{q}\sum_{r=1}^{q}\overline{\chi}(r)e^{-\frac{2\pi\mathrm{i}k\overline{a}r}{(a,b)q}}
\sideset{}{'}\sum_{d=-\infty}^{+\infty}\frac{e^{\frac{2\pi\mathrm{i}d(by-br-aqz)}{(a,b)q}}}{\bigl(d+\frac{k\overline{a}}{b}\bigl)^{j}}\nonumber\\
&&\qquad=\frac{(-1)^{j-1}(2\pi \mathrm{i})^{j}\mathrm{sgn}(b)(a,b)^{1-j}G(1,\chi)}{qj!b^{1-j}}\sum_{r=1}^{q}\overline{\chi}(r)\nonumber\\
&&\qquad\qquad\times
\sum_{l=1}^{|\frac{b}{(a,b)}|}e^{\frac{2\pi \mathrm{i}(l-\frac{by-aqz}{(a,b)q})k\overline{a}}{b}}\overline{B}_{j}\biggl(\frac{(a,b)l-\frac{by-br-aqz}{q}}{b}\biggl),
\end{eqnarray*}
and
\begin{eqnarray*}
&&\sideset{}{'}\sum_{d=-\infty}^{+\infty}\frac{\chi(\frac{k\overline{a}}{(a,b)}+\frac{bd}{(a,b)})e^{\frac{2\pi\mathrm{i}d(by-aqz)}{(a,b)q}}}{\bigl(d-\frac{k\overline{b}}{a}\bigl)^{j}}\nonumber\\
&&\qquad=\frac{G(1,\chi)}{q}\sum_{r=1}^{q}\overline{\chi}(r)e^{-\frac{2\pi\mathrm{i}k\overline{a}r}{(a,b)q}}
\sideset{}{'}\sum_{d=-\infty}^{+\infty}\frac{e^{\frac{2\pi\mathrm{i}d(by-br-aqz)}{(a,b)q}}}{\bigl(d-\frac{k\overline{b}}{a}\bigl)^{j}}\nonumber\\
&&\qquad=\frac{(-1)^{j-1}(2\pi \mathrm{i})^{j}\mathrm{sgn}(a)(a,b)^{1-j}G(1,\chi)}{qj!a^{1-j}}\sum_{r=1}^{q}\overline{\chi}(r)e^{-\frac{2\pi\mathrm{i}kr}{aq}}\nonumber\\
&&\qquad\qquad\times
\sum_{l=1}^{|\frac{a}{(a,b)}|}e^{-\frac{2\pi \mathrm{i}(l-\frac{by-aqz}{(a,b)q})k\overline{b}}{a}}\overline{B}_{j}\biggl(\frac{(a,b)l-\frac{by-br-aqz}{q}}{a}\biggl),
\end{eqnarray*}
as desired.
\end{proof}

\begin{lemma}\label{lem2.6} Let $m,n\in\mathbb{N}$, and $d,x,y\in\mathbb{R}$ be distinct. Then
\begin{eqnarray}
&&\frac{1}{(d-x)^{m}(d-y)^{n}}\nonumber\\
&&\qquad=\sum_{j=1}^{m}\binom{m+n-j-1}{n-1}\frac{(-1)^{m-j}}{(x-y)^{m+n-j}(d-x)^{j}}\nonumber\\
&&\qquad\quad+\sum_{j=1}^{n}\binom{m+n-j-1}{m-1}\frac{(-1)^{n-j}}{(y-x)^{m+n-j}(d-y)^{j}}.
\end{eqnarray}
\end{lemma}

\begin{proof}
See \cite[Lemma 2.2]{he1} for details.
\end{proof}

\begin{lemma}\label{lem2.7} Let $m,n\in\mathbb{N}$. Then
\begin{equation}
\sum_{j=1}^{m}\binom{m+n-j-1}{n-1}=\binom{m+n-1}{n}.
\end{equation}
\end{lemma}

\begin{proof}
See \cite[Lemma 2.8]{he1} for details.
\end{proof}

\section{Proof of Theorems}\label{sec3}

\noindent{\em{The proof of Theorem \ref{thm1.1}.}} Let
\begin{equation*}
f(x)=\overline{B}_{m}(ax+y,\overline{\chi})\overline{B}_{n}\biggl(\frac{bx}{q}+z\biggl).
\end{equation*}
Since $f(x)$ is of bounded variation on every finite interval, by Dirichlet-Jordan test, $f(x)$ may be expanded in a Fourier series:
\begin{equation}\label{eq3.1}
\frac{f(x^{+})+f(x^{-})}{2}=\sum_{k=-\infty}^{+\infty}c_{k}(f)e^{\frac{2\pi\mathrm{i}kx}{q}},
\end{equation}
where the Fourier coefficients $c_{k}(f)$ are determined by
\begin{equation}\label{eq3.2}
c_{k}(f)=\frac{1}{q}\int_{0}^{q}\overline{B}_{m}(ax+y,\overline{\chi})\overline{B}_{n}\biggl(\frac{bx}{q}+z\biggl)e^{-\frac{2\pi\mathrm{i}kx}{q}}dx.
\end{equation}
Using Parseval's formula from \cite[Proposition 3.1.10]{cohen} or \cite[p. 81]{stein}, together with \eqref{eq2.5} and \eqref{eq2.6}, we may rewrite \eqref{eq3.2} as
\begin{equation}\label{eq3.3}
c_{k}(f)=\frac{q^{m-1}m!n!G(1,\overline{\chi})}{(2\pi \mathrm{i})^{m+n}}\underset{al+bj=k}{\sideset{}{'}\sum_{l=-\infty}^{+\infty}\sideset{}{'}\sum_{j=-\infty}^{+\infty}}\frac{\chi(l)e^{2\pi\mathrm{i}(\frac{ly}{q}+jz)}} {l^{m}j^{n}}.
\end{equation}
Since the Diophantine equation $ax+by=k$ admits solutions precisely when $(a,b)\mid k$, any integer solution may be written as
\begin{equation*}
x=\frac{k\overline{a}}{(a,b)}+\frac{bd}{(a,b)},\quad y=\frac{k\overline{b}}{(a,b)}-\frac{ad}{(a,b)},
\end{equation*}
where $\overline{a},\overline{b},d\in\mathbb{Z}$ satisfy $a\overline{a}+b\overline{b}=(a,b)$. It then follows from \eqref{eq3.3} that, if $(a,b)\nmid k$, then
\begin{equation}\label{eq3.4}
c_{k}(f)=0;
\end{equation}
and if $(a,b)\mid k$, then
\begin{eqnarray}\label{eq3.5}
c_{k}(f)&=&\frac{q^{m-1}m!n!(a,b)^{m+n}G(1,\overline{\chi})e^{\frac{2\pi\mathrm{i}k(\overline{a}y+\overline{b}qz)}{(a,b)q}}}{(2\pi \mathrm{i})^{m+n}}\nonumber\\
&&\quad\times\sideset{}{'}\sum_{d=-\infty}^{+\infty}\frac{\chi(\frac{k\overline{a}}{(a,b)}+\frac{bd}{(a,b)})e^{\frac{2\pi\mathrm{i}d(by-aqz)}{(a,b)q}}}{(k\overline{a}+bd)^{m}(k\overline{b}-ad)^{n}}\nonumber\\
&=&\frac{(-1)^{n}q^{m-1}m!n!(a,b)^{m+n}G(1,\overline{\chi})e^{\frac{2\pi\mathrm{i}k(\overline{a}y+\overline{b}qz)}{(a,b)q}}}{a^{n}b^{m}(2\pi \mathrm{i})^{m+n}}\nonumber\\
&&\quad\times\sideset{}{'}\sum_{d=-\infty}^{+\infty}\frac{\chi(\frac{k\overline{a}}{(a,b)}+\frac{bd}{(a,b)})e^{\frac{2\pi\mathrm{i}d(by-aqz)}{(a,b)q}}}{\bigl(d+\frac{k\overline{a}}{b}\bigl)^{m}\bigl(d-\frac{k\overline{b}}{a}\bigl)^{n}}.
\end{eqnarray}
In particular, we see from \eqref{eq2.5} and \eqref{eq3.5} that
\begin{equation}\label{eq3.6}
c_{0}(f)=\frac{(-1)^{n-1}m!n!(a,b)^{m+n}\chi(\frac{b}{(a,b)})\overline{B}_{m+n}\bigl(\frac{by-aqz}{(a,b)},\overline{\chi}\bigl)}{(aq)^{n}b^{m}(m+n)!}.
\end{equation}
We next consider the case $k\not=0$ in \eqref{eq3.5}. Note that there is an implicit restriction
\begin{equation*}
\biggl(\frac{a}{(a,b)},\overline{b}\biggl)=\biggl(\frac{b}{(a,b)},\overline{a}\biggl)=1.
\end{equation*}
Hence, by taking $x=-k\overline{a}/b$ and $y=k\overline{b}/a$ in Lemma \ref{lem2.6}, we obtain from \eqref{eq3.5} that if $(a,b)\mid k$, then
\begin{eqnarray}\label{eq3.7}
c_{k}(f)&=&\frac{(-1)^{n}q^{m-1}m!n!(a,b)^{m+n}G(1,\overline{\chi})e^{\frac{2\pi\mathrm{i}k(\overline{a}y+\overline{b}qz)}{(a,b)q}}}{a^{n}b^{m}(2\pi \mathrm{i})^{m+n}}\nonumber\\
&&\quad\times\biggl(\sum_{j=1}^{m}\binom{m+n-j-1}{n-1}(-1)^{m-j}\biggl(-\frac{ab}{k(a,b)}\biggl)^{m+n-j}\nonumber\\
&&\qquad\quad\times\biggl(\sideset{}{'}\sum_{d=-\infty}^{+\infty}\frac{\chi(\frac{k\overline{a}}{(a,b)}+\frac{bd}{(a,b)})e^{\frac{2\pi\mathrm{i}d(by-aqz)}{(a,b)q}}}{\bigl(d+\frac{k\overline{a}}{b}\bigl)^{j}}\nonumber\\
&&\qquad\qquad-\epsilon(a,k)\chi\biggl(\frac{k}{a}\biggl)e^{\frac{2\pi\mathrm{i}k\overline{b}(by-aqz)}{a(a,b)q}}\biggl(\frac{ab}{k(a,b)}\biggl)^{j}\biggl)\nonumber\\
&&\qquad+\sum_{j=1}^{n}\binom{m+n-j-1}{m-1}(-1)^{n-j}\biggl(\frac{ab}{k(a,b)}\biggl)^{m+n-j}\nonumber\\
&&\qquad\quad\times\biggl(\sideset{}{'}\sum_{d=-\infty}^{+\infty}\frac{\chi(\frac{k\overline{a}}{(a,b)}+\frac{bd}{(a,b)})e^{\frac{2\pi\mathrm{i}d(by-aqz)}{(a,b)q}}}{\bigl(d-\frac{k\overline{b}}{a}\bigl)^{j}}\nonumber\\
&&\qquad\qquad-\epsilon(b,k)\chi(0)e^{-\frac{2\pi\mathrm{i}k\overline{a}(by-aqz)}{b(a,b)q}}\biggl(-\frac{ab}{k(a,b)}\biggl)^{j}\biggl)\biggl),
\end{eqnarray}
where $\epsilon(a,k)=1$ if $a\mid k$ and $0$ otherwise. From \eqref{eq2.4} and the geometric sum in \cite[Theorem 8.1]{apostol}, we see that
\begin{eqnarray}\label{eq3.8}
G(1,\chi)G(1,\overline{\chi})&=&G(1,\chi)\sum_{r=1}^{q}\overline{\chi}(r)e^{\frac{2\pi \mathrm{i}r}{q}}\nonumber\\
&=&\sum_{r=1}^{q}G(r,\chi)e^{\frac{2\pi \mathrm{i}r}{q}}\nonumber\\
&=&\sum_{j=1}^{q}\chi(j)\sum_{r=1}^{q}e^{\frac{2\pi \mathrm{i}(j+1)r}{q}}\nonumber\\
&=&\chi(-1)q.
\end{eqnarray}
Using Lemma \ref{lem2.5} and \eqref{eq3.8}, we find that \eqref{eq3.7} can be rewritten as
\begin{eqnarray*}
c_{k}(f)&=&-q^{m-1}b^{n-1}m!n!(a,b)\chi(-1)\mathrm{sgn}(b)\sum_{j=1}^{m}\binom{m+n-j-1}{n-1}\frac{(-1)^{j}a^{m-j}}{j!(2\pi\mathrm{i}k)^{m+n-j}}\nonumber\\
&&\quad\times\sum_{r=1}^{q}\overline{\chi}(r)\sum_{l=1}^{|\frac{b}{(a,b)}|}e^{\frac{2\pi\mathrm{i}k(z+\overline{a}l)}{b}}\overline{B}_{j}\biggl(\frac{(a,b)l-\frac{by-br-aqz}{q}}{b}\biggl)\nonumber\\
&&-\epsilon(a,k)\frac{q^{m-1}a^{m}b^{n}m!n!\chi(\frac{k}{a})G(1,\overline{\chi})e^{\frac{2\pi\mathrm{i}ky}{aq}}}{(2\pi\mathrm{i}k)^{m+n}}\sum_{j=1}^{m}\binom{m+n-j-1}{n-1}\nonumber\\
&&-(aq)^{m-1}m!n!(a,b)\chi(-1)\mathrm{sgn}(a)\sum_{j=1}^{n}\binom{m+n-j-1}{m-1}\frac{b^{n-j}}{j!(2\pi\mathrm{i}k)^{m+n-j}}\nonumber\\
&&\quad\times\sum_{r=1}^{q}\overline{\chi}(r)e^{-\frac{2\pi\mathrm{i}kr}{aq}}\sum_{l=1}^{|\frac{a}{(a,b)}|}e^{\frac{2\pi\mathrm{i}k(\frac{y}{q}-\overline{b}l)}{a}}\overline{B}_{j}\biggl(\frac{(a,b)l-\frac{by-br-aqz}{q}}{a}\biggl)\nonumber\\
&&-\epsilon(b,k)\frac{q^{m-1}a^{m}b^{n}m!n!\chi(0)G(1,\overline{\chi})e^{\frac{2\pi\mathrm{i}kz}{b}}}{(2\pi\mathrm{i}k)^{m+n}}\sum_{j=1}^{n}\binom{m+n-j-1}{m-1}.
\end{eqnarray*}
This, together with Lemma \ref{lem2.7}, yields that if $(a,b)\mid k$, then
\begin{eqnarray}\label{eq3.9}
c_{k}(f)&=&-nq^{m-1}b^{n-1}(a,b)\chi(-1)\mathrm{sgn}(b)\sum_{j=1}^{m}\binom{m}{j}\frac{(-1)^{j}a^{m-j}(m+n-j-1)!}{(2\pi\mathrm{i}k)^{m+n-j}}\nonumber\\
&&\quad\times\sum_{r=1}^{q}\overline{\chi}(r)\sum_{l=1}^{|\frac{b}{(a,b)}|}e^{\frac{2\pi\mathrm{i}k(z+\overline{a}l)}{b}}\overline{B}_{j}\biggl(\frac{(a,b)l-\frac{by-br-aqz}{q}}{b}\biggl)\nonumber\\
&&-m(aq)^{m-1}(a,b)\chi(-1)\mathrm{sgn}(a)\sum_{j=1}^{n}\binom{n}{j}\frac{b^{n-j}(m+n-j-1)!}{(2\pi\mathrm{i}k)^{m+n-j}}\nonumber\\
&&\quad\times\sum_{r=1}^{q}\overline{\chi}(r)e^{-\frac{2\pi\mathrm{i}kr}{aq}}\sum_{l=1}^{|\frac{a}{(a,b)}|}e^{\frac{2\pi\mathrm{i}k(\frac{y}{q}-\overline{b}l)}{a}}\overline{B}_{j}\biggl(\frac{(a,b)l-\frac{by-br-aqz}{q}}{a}\biggl)\nonumber\\
&&-\epsilon(a,k)e^{\frac{2\pi\mathrm{i}ky}{aq}}\frac{m(m+n-1)!q^{m-1}a^{m}b^{n}\chi(\frac{k}{a})G(1,\overline{\chi})}{(2\pi\mathrm{i}k)^{m+n}}\nonumber\\
&&-\epsilon(b,k)e^{\frac{2\pi\mathrm{i}kz}{b}}\frac{n(m+n-1)!q^{m-1}a^{m}b^{n}\chi(0)G(1,\overline{\chi})}{(2\pi\mathrm{i}k)^{m+n}}.
\end{eqnarray}
Since $\bigl(\frac{a}{(a,b)},\frac{b}{(a,b)}\bigl)=1$, $\frac{al}{(a,b)}$ runs over a complete residue
system modulo $|\frac{b}{(a,b)}|$ as $l$ does, and $\frac{-bl}{(a,b)}$ runs over a complete residue system modulo $|\frac{a}{(a,b)}|$
as $l$ does. It follows from \eqref{eq2.7} and \eqref{eq3.9} that, if $(a,b)\mid k$, then
\begin{eqnarray}\label{eq3.10}
c_{k}(f)&=&-nq^{m-1}b^{n-1}(a,b)\chi(-1)\mathrm{sgn}(b)\sum_{j=1}^{m}\binom{m}{j}\frac{(-1)^{j}a^{m-j}(m+n-j-1)!}{(2\pi\mathrm{i}k)^{m+n-j}}\nonumber\\
&&\quad\times\sum_{r=1}^{q}\overline{\chi}(r)\sum_{l=1}^{|\frac{b}{(a,b)}|}e^{\frac{2\pi\mathrm{i}k(l+z)}{b}}\overline{B}_{j}\biggl(\frac{a(l+z)}{b}-\frac{y-r}{q}\biggl)\nonumber\\
&&-m(aq)^{m-1}(a,b)\chi(-1)\mathrm{sgn}(a)\sum_{j=1}^{n}\binom{n}{j}\frac{(-1)^{j}b^{n-j}(m+n-j-1)!}{(2\pi\mathrm{i}k)^{m+n-j}}\nonumber\\
&&\quad\times\sum_{r=1}^{q}\overline{\chi}(r)e^{-\frac{2\pi\mathrm{i}kr}{aq}}\sum_{l=1}^{|\frac{a}{(a,b)}|}e^{\frac{2\pi\mathrm{i}k(l+\frac{y}{q})}{a}}\overline{B}_{j}\biggl(\frac{bl+\frac{by-br-aqz}{q}}{a}\biggl)\nonumber\\
&&-\epsilon(a,k)e^{\frac{2\pi\mathrm{i}ky}{aq}}\frac{m(m+n-1)!q^{m-1}a^{m}b^{n}\chi(\frac{k}{a})G(1,\overline{\chi})}{(2\pi\mathrm{i}k)^{m+n}}\nonumber\\
&&-\epsilon(b,k)e^{\frac{2\pi\mathrm{i}kz}{b}}\frac{n(m+n-1)!q^{m-1}a^{m}b^{n}\chi(0)G(1,\overline{\chi})}{(2\pi\mathrm{i}k)^{m+n}}.
\end{eqnarray}
Inserting \eqref{eq3.4}, \eqref{eq3.6} and \eqref{eq3.10} into \eqref{eq3.1}, we arrive at
\begin{eqnarray}\label{eq3.11}
&&\overline{B}_{m}(ax+y,\overline{\chi})\overline{B}_{n}\biggl(\frac{bx}{q}+z\biggl)+\delta_{1,m}\delta_{1,n}\mathrm{sgn}(ab)\frac{\delta_{\mathbb{Z}}(ax+y)\delta_{\mathbb{Z}}(\frac{bx}{q}+z)\overline{\chi}(-ax-y)}{4}\nonumber\\
&&\qquad=\sideset{}{'}\sum_{\substack{k=-\infty\\(a,b)\mid k}}^{+\infty}\biggl(-nq^{m-1}b^{n-1}(a,b)\chi(-1)\mathrm{sgn}(b)\sum_{j=1}^{m}\binom{m}{j}\frac{(-1)^{j}a^{m-j}}{(2\pi\mathrm{i}k)^{m+n-j}}\nonumber\\
&&\qquad\qquad\times(m+n-j-1)!\sum_{r=1}^{q}\overline{\chi}(r)\sum_{l=1}^{|\frac{b}{(a,b)}|}e^{\frac{2\pi\mathrm{i}k(l+z)}{b}}\overline{B}_{j}\biggl(\frac{a(l+z)}{b}-\frac{y-r}{q}\biggl)\nonumber\\
&&\qquad\quad-m(aq)^{m-1}(a,b)\chi(-1)\mathrm{sgn}(a)\sum_{j=1}^{n}\binom{n}{j}\frac{(-1)^{j}b^{n-j}(m+n-j-1)!}{(2\pi\mathrm{i}k)^{m+n-j}}\nonumber\\
&&\qquad\qquad\times\sum_{r=1}^{q}\overline{\chi}(r)e^{-\frac{2\pi\mathrm{i}kr}{aq}}\sum_{l=1}^{|\frac{a}{(a,b)}|}e^{\frac{2\pi\mathrm{i}k(l+\frac{y}{q})}{a}}
\overline{B}_{j}\biggl(\frac{bl+\frac{by-br}{q}}{a}-z\biggl)\biggl)e^{\frac{2\pi\mathrm{i}kx}{q}}\nonumber\\
&&\qquad\quad-\sideset{}{'}\sum_{\substack{k=-\infty\\a\mid k}}^{+\infty}\biggl(\frac{m(m+n-1)!q^{m-1}a^{m}b^{n}\chi(\frac{k}{a})G(1,\overline{\chi})e^{\frac{2\pi\mathrm{i}ky}{aq}}}{(2\pi\mathrm{i}k)^{m+n}}\biggl)e^{\frac{2\pi\mathrm{i}kx}{q}}\nonumber\\
&&\qquad\quad-\sideset{}{'}\sum_{\substack{k=-\infty\\b\mid k}}^{+\infty}\biggl(\frac{n(m+n-1)!q^{m-1}a^{m}b^{n}\chi(0)G(1,\overline{\chi})e^{\frac{2\pi\mathrm{i}kz}{b}}}{(2\pi\mathrm{i}k)^{m+n}}\biggl)e^{\frac{2\pi\mathrm{i}kx}{q}}\nonumber\\
&&\qquad\quad+\frac{(-1)^{n-1}m!n!(a,b)^{m+n}\chi(\frac{b}{(a,b)})\overline{B}_{m+n}\bigl(\frac{by-aqz}{(a,b)},\overline{\chi}\bigl)}{(aq)^{n}b^{m}(m+n)!}.
\end{eqnarray}
Thus, from \eqref{eq2.5}, \eqref{eq2.6} and \eqref{eq3.11}, we obtain
\begin{eqnarray}\label{eq3.12}
&&\overline{B}_{m}(ax+y,\overline{\chi})\overline{B}_{n}\biggl(\frac{bx}{q}+z\biggl)+\delta_{1,m}\delta_{1,n}\mathrm{sgn}(ab)\frac{\delta_{\mathbb{Z}}(ax+y)\delta_{\mathbb{Z}}(\frac{bx}{q}+z)\overline{\chi}(-ax-y)}{4}\nonumber\\
&&\qquad=nq^{m-1}b^{n-1}\chi(-1)\mathrm{sgn}(b)\sum_{j=1}^{m}\binom{m}{j}\frac{(-1)^{j}a^{m-j}}{(a,b)^{m+n-j-1}(m+n-j)}\nonumber\\
&&\qquad\qquad\times\sum_{r=1}^{q}\overline{\chi}(r)\sum_{l=1}^{|\frac{b}{(a,b)}|}\overline{B}_{j}\biggl(\frac{a(l+z)}{b}-\frac{y-r}{q}\biggl)\nonumber\\
&&\qquad\qquad\times\overline{B}_{m+n-j}\biggl(\frac{(a,b)x}{q}+\frac{(a,b)(l+z)}{b}\biggl)\nonumber\\
&&\qquad\quad+m(aq)^{m-1}\chi(-1)\mathrm{sgn}(a)\sum_{j=1}^{n}\binom{n}{j}\frac{(-1)^{j}b^{n-j}}{(a,b)^{m+n-j-1}(m+n-j)}\nonumber\\
&&\qquad\qquad\times\sum_{r=1}^{q}\overline{\chi}(r)\sum_{l=1}^{|\frac{a}{(a,b)}|}\overline{B}_{j}\biggl(\frac{bl+\frac{by-br}{q}}{a}-z\biggl)\nonumber\\
&&\qquad\qquad\times\overline{B}_{m+n-j}\biggl(\frac{(a,b)x}{q}+\frac{(a,b)(l+\frac{y}{q})}{a}-\frac{(a,b)r}{aq}\biggl)\nonumber\\
&&\qquad\quad+\frac{mb^{n}\overline{B}_{m+n}(ax+y,\overline{\chi})}{(aq)^{n}(m+n)}+\frac{nq^{m-1}a^{m}\chi(0)G(1,\overline{\chi})\overline{B}_{m+n}(\frac{bx}{q}+z)}{b^{m}(m+n)}\nonumber\\
&&\qquad\quad+\frac{(-1)^{n-1}m!n!(a,b)^{m+n}\chi(\frac{b}{(a,b)})\overline{B}_{m+n}\bigl(\frac{by-aqz}{(a,b)},\overline{\chi}\bigl)}{(aq)^{n}b^{m}(m+n)!}.
\end{eqnarray}
It can be shown from \eqref{eq2.14}, the property of residue systems, and the division algorithm that
\begin{eqnarray}\label{eq3.13}
&&\frac{1}{(a,b)^{m+n-j-1}}\sum_{l=1}^{|\frac{b}{(a,b)}|}\overline{B}_{j}\biggl(\frac{a(l+z)}{b}-\frac{y-r}{q}\biggl)\nonumber\\
&&\qquad\times\overline{B}_{m+n-j}\biggl(\frac{(a,b)x}{q}+\frac{(a,b)(l+z)}{b}\biggl)\nonumber\\
&&\qquad=\sum_{k=0}^{(a,b)-1}\sum_{l=1}^{|\frac{b}{(a,b)}|}\overline{B}_{j}\biggl(\frac{a\bigl(l+z+k|\frac{b}{(a,b)}|\bigl)}{b}-\frac{y-r}{q}\biggl)\nonumber\\
&&\qquad\qquad\times\overline{B}_{m+n-j}\biggl(\frac{x}{q}+\frac{l+z+k|\frac{b}{(a,b)}|}{b}\biggl)\nonumber\\
&&\qquad=\sum_{l=1}^{|b|}\overline{B}_{j}\biggl(\frac{a(l+z)}{b}-\frac{y-r}{q}\biggl)\overline{B}_{m+n-j}\biggl(\frac{x}{q}+\frac{l+z}{b}\biggl),
\end{eqnarray}
and
\begin{eqnarray}\label{eq3.14}
&&\frac{1}{(a,b)^{m+n-j-1}}\sum_{l=1}^{|\frac{a}{(a,b)}|}\overline{B}_{j}\biggl(\frac{bl+\frac{by-br}{q}}{a}-z\biggl)\nonumber\\
&&\qquad\times\overline{B}_{m+n-j}\biggl(\frac{(a,b)x}{q}+\frac{(a,b)(l+\frac{y}{q})}{a}-\frac{(a,b)r}{aq}\biggl)\nonumber\\
&&\qquad=\sum_{k=0}^{(a,b)-1}\sum_{l=1}^{|\frac{a}{(a,b)}|}\overline{B}_{j}\biggl(\frac{b(l+k|\frac{a}{(a,b)}|)+\frac{by-br}{q}}{a}-z\biggl)\nonumber\\
&&\qquad\qquad\times\overline{B}_{m+n-j}\biggl(\frac{x}{q}+\frac{l+k|\frac{a}{(a,b)}|+\frac{y}{q}}{a}-\frac{r}{aq}\biggl)\nonumber\\
&&\qquad=\sum_{l=1}^{|a|}\overline{B}_{j}\biggl(\frac{bl}{a}+\frac{by-br}{aq}-z\biggl)\overline{B}_{m+n-j}\biggl(\frac{x}{q}+\frac{l}{a}+\frac{y-r}{aq}\biggl).
\end{eqnarray}
Therefore, applying \eqref{eq3.13} and \eqref{eq3.14} to the right-hand side of \eqref{eq3.12}, we obtain \eqref{eq1.11}, which completes the proof of Theorem \ref{thm1.1}.

\noindent{\em{The proof of Theorem \ref{thm1.2}.}} It is easily seen from \eqref{eq2.7}, \eqref{eq2.8} and the property of residue systems that, for $j\in\mathbb{N}$,
\begin{eqnarray}\label{eq3.15}
&&\chi_{2}(-1)\sum_{r=1}^{q}\overline{\chi_{2}}(r)\overline{B}_{j}\biggl(\frac{a(l+z)}{b}-\frac{y-r}{q}\biggl)\nonumber\\
&&\qquad=(-1)^{j}\sum_{r=1}^{q}\overline{\chi_{2}}(-r)\overline{B}_{j}\biggl(\frac{-a(l+z)}{b}+\frac{y-r}{q}\biggl)\nonumber\\
&&\qquad=(-1)^{j}\sum_{r=1}^{q}\overline{\chi_{2}}(r)\overline{B}_{j}\biggl(\frac{-a(l+z)}{b}+\frac{y+r}{q}\biggl)\nonumber\\
&&\qquad=\frac{(-1)^{j}}{q^{j-1}}\overline{B}_{j}\biggl(\frac{-aq(l+z)}{b}+y,\overline{\chi_{2}}\biggl).
\end{eqnarray}
Similarly, it follows from \eqref{eq2.8}, \eqref{eq2.14} and the property of residue systems that
\begin{eqnarray}\label{eq3.16}
&&\chi_{2}(-1)\sum_{r=1}^{q}\overline{\chi_{2}}(r)\sum_{l=1}^{a}\overline{B}_{m+n}\biggl(\frac{k+x}{cq}+\frac{l}{a}+\frac{y-r}{aq}\biggl)\nonumber\\
&&\qquad=\frac{1}{a^{m+n-1}}\sum_{r=1}^{q}\overline{\chi_{2}}(-r)\overline{B}_{m+n}\biggl(\frac{a(k+x)}{cq}+\frac{y-r}{q}\biggl)\nonumber\\
&&\qquad=\frac{1}{a^{m+n-1}}\sum_{r=1}^{q}\overline{\chi_{2}}(r)\overline{B}_{m+n}\biggl(\frac{a(k+x)}{cq}+\frac{y+r}{q}\biggl)\nonumber\\
&&\qquad=\frac{1}{(aq)^{m+n-1}}\overline{B}_{m+n}\biggl(\frac{a(k+x)}{c}+y,\overline{\chi_{2}}\biggl).
\end{eqnarray}
Hence, setting $\chi=\chi_{2}$, replacing $x$ by $(k+x)/c$, and summing both sides of \eqref{eq1.11} over $k=1,\ldots,cq$ with weight $\chi_{1}(k)$, we obtain from \eqref{eq3.15}, \eqref{eq3.16} and the orthogonality relations for characters in \cite[Theorem 6.10]{apostol} that
\begin{eqnarray}\label{eq3.17}
&&\sum_{k=1}^{cq}\chi_{1}(k)\overline{B}_{m}\biggl(\frac{a(k+x)}{c}+y,\overline{\chi_{2}}\biggl)\overline{B}_{n}\biggl(\frac{b(k+x)}{cq}+z\biggl)\nonumber\\
&&\qquad=nq^{m-1}b^{n-1}\sum_{j=1}^{m}\binom{m}{j}\frac{a^{m-j}}{q^{j-1}(m+n-j)}\sum_{l=1}^{b}\overline{B}_{j}\biggl(\frac{-aq(l+z)}{b}+y,\overline{\chi_{2}}\biggl)\nonumber\\
&&\qquad\qquad\times\sum_{k=1}^{cq}\chi_{1}(k)\overline{B}_{m+n-j}\biggl(\frac{k+x}{cq}+\frac{l+z}{b}\biggl)\nonumber\\
&&\qquad\quad+m(aq)^{m-1}\chi_{2}(-1)\sum_{j=0}^{n}\binom{n}{j}\frac{(-1)^{j}b^{n-j}}{m+n-j}\sum_{k=1}^{cq}\chi_{1}(k)\sum_{r=1}^{q}\overline{\chi_{2}}(r)\nonumber\\
&&\qquad\qquad\times\sum_{l=1}^{a}\overline{B}_{j}\biggl(\frac{bl}{a}+\frac{by-br}{aq}-z\biggl)\overline{B}_{m+n-j}\biggl(\frac{k+x}{cq}+\frac{l}{a}+\frac{y-r}{aq}\biggl)\nonumber\\
&&\qquad\quad-\frac{\delta_{1,m}\delta_{1,n}\chi_{2}(-1)}{4}\sum_{k=1}^{cq}\chi_{1}(k)\delta_{\mathbb{Z}}\biggl(\frac{a(k+x)}{c}+y\biggl)\delta_{\mathbb{Z}}\biggl(\frac{b(k+x)}{cq}+z\biggl)\nonumber\\
&&\qquad\qquad\times\chi_{2}\biggl(\frac{a(k+x)}{c}+y\biggl).
\end{eqnarray}
Note that from \eqref{eq2.8}, \eqref{eq2.14} and the division algorithm we have
\begin{eqnarray}\label{eq3.18}
&&\sum_{k=1}^{cq}\chi_{1}(k)\overline{B}_{m+n-j}\biggl(\frac{k+x}{cq}+\frac{l+z}{b}\biggl)\nonumber\\
&&\qquad=\sum_{l=0}^{c-1}\sum_{s=1}^{q}\chi_{1}(ql+s)\overline{B}_{m+n-j}\biggl(\frac{ql+s+x}{cq}+\frac{l+z}{b}\biggl)\nonumber\\
&&\qquad=\frac{1}{c^{m+n-j-1}}\sum_{s=1}^{q}\chi_{1}(s)\overline{B}_{m+n-j}\biggl(\frac{s+x}{q}+\frac{c(l+z)}{b}\biggl)\nonumber\\
&&\qquad=\frac{1}{(cq)^{m+n-j-1}}\overline{B}_{m+n-j}\biggl(\frac{cq(l+z)}{b}+x,\chi_{1}\biggl),
\end{eqnarray}
and
\begin{eqnarray*}
&&\sum_{k=1}^{cq}\chi_{1}(k)\sum_{r=1}^{q}\overline{\chi_{2}}(r)\sum_{l=1}^{a}\overline{B}_{j}\biggl(\frac{bl}{a}+\frac{by-br}{aq}-z\biggl)\overline{B}_{m+n-j}\biggl(\frac{k+x}{cq}+\frac{l}{a}+\frac{y-r}{aq}\biggl)\nonumber\\
&&\qquad=\sum_{t=0}^{c-1}\sum_{s=1}^{q}\chi_{1}(qt+s)\sum_{r=1}^{q}\overline{\chi_{2}}(r)\sum_{l=1}^{a}\overline{B}_{j}\biggl(\frac{bl}{a}+\frac{by-br}{aq}-z\biggl)\nonumber\\
&&\qquad\qquad\times\overline{B}_{m+n-j}\biggl(\frac{t}{c}+\frac{s+x}{cq}+\frac{l}{a}+\frac{y-r}{aq}\biggl)\nonumber\\
&&\qquad=\frac{1}{c^{m+n-j-1}}\sum_{s=1}^{q}\chi_{1}(s)\sum_{r=1}^{q}\overline{\chi_{2}}(r)\sum_{l=1}^{a}\overline{B}_{j}\biggl(\frac{bl}{a}+\frac{by-br}{aq}-z\biggl)\nonumber\\
&&\qquad\qquad\times\overline{B}_{m+n-j}\biggl(\frac{s+x}{q}+\frac{cl}{a}+\frac{c(y-r)}{aq}\biggl)\nonumber\\
&&\qquad=\frac{\chi_{2}(-1)}{(cq)^{m+n-j-1}}\sum_{r=1}^{q}\sum_{l=1}^{a}\overline{\chi_{2}}(ql-r)\overline{B}_{j}\biggl(\frac{bl}{a}+\frac{by-br}{aq}-z\biggl)\nonumber\\
&&\qquad\qquad\times\overline{B}_{m+n-j}\biggl(x+\frac{cql}{a}+\frac{c(y-r)}{a},\chi_{1}\biggl)\nonumber\\
&&\qquad=\frac{\chi_{2}(-1)}{(cq)^{m+n-j-1}}\sum_{k=1}^{aq}\overline{\chi_{2}}(k)\overline{B}_{j}\biggl(\frac{b(k+y)}{aq}-z\biggl)\overline{B}_{m+n-j}\biggl(\frac{c(k+y)}{a}+x,\chi_{1}\biggl).
\end{eqnarray*}
Thus, inserting the above two identities into \eqref{eq3.17} and then replacing $x$ by $z$, $y$ by $x$, and $z$ by $y$, we obtain \eqref{eq1.12} and complete the proof of Theorem \ref{thm1.2}.

\noindent{\em{The proof of Theorem \ref{thm1.3}.}} Replacing $x$ by $(k+x)/c$, and summing both sides of \eqref{eq1.11} over $k=1,\ldots,cq$, we see from \eqref{eq2.14}, \eqref{eq3.15} and \eqref{eq3.16} that
\begin{eqnarray}\label{eq3.19}
&&\sum_{k=1}^{cq}\overline{B}_{m}\biggl(\frac{a(k+x)}{c}+y,\overline{\chi}\biggl)\overline{B}_{n}\biggl(\frac{b(k+x)}{cq}+z\biggl)\nonumber\\
&&\qquad=n\biggl(\frac{b}{cq}\biggl)^{n-1}\sum_{j=1}^{m}\binom{m}{j}\frac{(\frac{a}{c})^{m-j}}{m+n-j}\sum_{l=1}^{b}\overline{B}_{j}\biggl(\frac{-aq(l+z)}{b}+y,\overline{\chi}\biggl)\nonumber\\
&&\qquad\qquad\times\overline{B}_{m+n-j}\biggl(\frac{cq(l+z)}{b}+x\biggl)\nonumber\\
&&\qquad\quad+m\biggl(\frac{a}{c}\biggl)^{m-1}\chi(-1)\sum_{j=0}^{n}\binom{n}{j}\frac{(-1)^{j}(\frac{b}{cq})^{n-j}}{m+n-j}\sum_{r=1}^{q}\overline{\chi}(r)\nonumber\\
&&\qquad\qquad\times\sum_{l=1}^{a}\overline{B}_{j}\biggl(\frac{b(ql-r+y)}{aq}-z\biggl)\overline{B}_{m+n-j}\biggl(\frac{c(ql-r+y)}{a}+x\biggl)\nonumber\\
&&\qquad\quad+\frac{(-1)^{n-1}m!n!(a,b)^{m+n}c\chi(\frac{b}{(a,b)})\overline{B}_{m+n}\bigl(\frac{by-aqz}{(a,b)},\overline{\chi}\bigl)}{a^{n}q^{n-1}b^{m}(m+n)!}\nonumber\\
&&\qquad\quad-\frac{\delta_{1,m}\delta_{1,n}\chi(-1)}{4}\sum_{k=1}^{cq}\delta_{\mathbb{Z}}\biggl(\frac{a(k+x)}{c}+y\biggl)\delta_{\mathbb{Z}}\biggl(\frac{b(k+x)}{cq}+z\biggl)\nonumber\\
&&\qquad\qquad\times\overline{\chi}\biggl(\frac{a(k+x)}{c}+y\biggl).
\end{eqnarray}
Observe that
\begin{eqnarray}\label{eq3.20}
&&\chi(-1)\sum_{r=1}^{q}\overline{\chi}(r)\sum_{l=1}^{a}\overline{B}_{j}\biggl(\frac{b(ql-r+y)}{aq}-z\biggl)\overline{B}_{m+n-j}\biggl(\frac{c(ql-r+y)}{a}+x\biggl)\nonumber\\
&&\qquad=\sum_{l=1}^{aq}\overline{\chi}(l)\overline{B}_{j}\biggl(\frac{b(l+y)}{aq}-z\biggl)\overline{B}_{m+n-j}\biggl(\frac{c(l+y)}{a}+x\biggl).
\end{eqnarray}
Now, inserting \eqref{eq3.20} into \eqref{eq3.19} and then replacing $x$ by $z$, $y$ by $x$, $z$ by $y$, and $\overline{\chi}$ by $\chi$, we obtain \eqref{eq1.16} and finish the proof of Theorem \ref{thm1.3}.

\noindent{\em{The proof of Theorem \ref{thm1.4}.}} Replacing $q$ by $q_{2}$ in \eqref{eq3.15}, we obtain that, for $j\in\mathbb{N}$,
\begin{eqnarray}\label{eq3.21}
&&\chi_{2}(-1)\sum_{r=1}^{q_{2}}\overline{\chi_{2}}(r)\overline{B}_{j}\biggl(\frac{a(l+z)}{b}-\frac{y-r}{q_{2}}\biggl)\nonumber\\
&&\qquad=\frac{(-1)^{j}}{q_{2}^{j-1}}\overline{B}_{j}\biggl(\frac{-aq_{2}(l+z)}{b}+y,\overline{\chi_{2}}\biggl).
\end{eqnarray}
Similarly, replacing $q$ by $q_{2}$, and $c$ by $cq_{1}$ in \eqref{eq3.16}, we obtain
\begin{eqnarray}\label{eq3.22}
&&\chi_{2}(-1)\sum_{r=1}^{q_{2}}\overline{\chi_{2}}(r)\sum_{l=1}^{a}\overline{B}_{m+n}\biggl(\frac{k+x}{cq_{1}q_{2}}+\frac{l}{a}+\frac{y-r}{aq_{2}}\biggl)\nonumber\\
&&\qquad=\frac{1}{(aq_{2})^{m+n-1}}\overline{B}_{m+n}\biggl(\frac{a(k+x)}{cq_{1}}+y,\overline{\chi_{2}}\biggl).
\end{eqnarray}
Hence, setting $\chi=\chi_{2}$, replacing $x$ by $(k+x)/cq_{1}$, and summing both sides of \eqref{eq1.11} over $k=1,\ldots,cq_{1}q_{2}$ with weight $\chi_{1}(k)$, we find from \eqref{eq3.21}, \eqref{eq3.22} and the orthogonality relations for characters that
\begin{eqnarray}\label{eq3.23}
&&\sum_{k=1}^{cq_{1}q_{2}}\chi_{1}(k)\overline{B}_{m}\biggl(\frac{a(k+x)}{cq_{1}}+y,\overline{\chi_{2}}\biggl)\overline{B}_{n}\biggl(\frac{b(k+x)}{cq_{1}q_{2}}+z\biggl)\nonumber\\
&&\qquad=nq_{2}^{m-1}b^{n-1}\sum_{j=1}^{m}\binom{m}{j}\frac{a^{m-j}}{q_{2}^{j-1}(m+n-j)}\sum_{l=1}^{b}\overline{B}_{j}\biggl(\frac{-aq_{2}(l+z)}{b}+y,\overline{\chi_{2}}\biggl)\nonumber\\
&&\qquad\qquad\times\sum_{k=1}^{cq_{1}q_{2}}\chi_{1}(k)\overline{B}_{m+n-j}\biggl(\frac{k+x}{cq_{1}q_{2}}+\frac{l+z}{b}\biggl)\nonumber\\
&&\qquad\quad+m(aq_{2})^{m-1}\chi_{2}(-1)\sum_{j=0}^{n}\binom{n}{j}\frac{(-1)^{j}b^{n-j}}{m+n-j}\sum_{k=1}^{cq_{1}q_{2}}\chi_{1}(k)\sum_{r=1}^{q_{2}}\overline{\chi_{2}}(r)\nonumber\\
&&\qquad\qquad\times\sum_{l=1}^{a}\overline{B}_{j}\biggl(\frac{bl}{a}+\frac{by-br}{aq_{2}}-z\biggl)\overline{B}_{m+n-j}\biggl(\frac{k+x}{cq_{1}q_{2}}+\frac{l}{a}+\frac{y-r}{aq_{2}}\biggl)\nonumber\\
&&\qquad\quad-\frac{\delta_{1,m}\delta_{1,n}\chi_{2}(-1)}{4}\sum_{k=1}^{cq_{1}q_{2}}\chi_{1}(k)\delta_{\mathbb{Z}}\biggl(\frac{a(k+x)}{cq_{1}}+y\biggl)
\delta_{\mathbb{Z}}\biggl(\frac{b(k+x)}{cq_{1}q_{2}}+z\biggl)\nonumber\\
&&\qquad\qquad\times\overline{\chi_{2}}\biggl(\frac{a(k+x)}{cq_{1}}+y\biggl).
\end{eqnarray}
Note that replacing $q$ by $q_{1}$ and $c$ by $cq_{2}$ in \eqref{eq3.18} yields
\begin{eqnarray}\label{eq3.24}
&&\sum_{k=1}^{cq_{1}q_{2}}\chi_{1}(k)\overline{B}_{m+n-j}\biggl(\frac{k+x}{cq_{1}q_{2}}+\frac{l+z}{b}\biggl)\nonumber\\
&&\qquad=\frac{1}{(cq_{1}q_{2})^{m+n-j-1}}\overline{B}_{m+n-j}\biggl(\frac{cq_{1}q_{2}(l+z)}{b}+x,\chi_{1}\biggl);
\end{eqnarray}
and using \eqref{eq2.8}, \eqref{eq2.14} and the division algorithm, we obtain
\begin{eqnarray}\label{eq3.25}
&&\sum_{k=1}^{cq_{1}q_{2}}\chi_{1}(k)\sum_{r=1}^{q_{2}}\overline{\chi_{2}}(r)\sum_{l=1}^{a}\overline{B}_{j}\biggl(\frac{bl}{a}+\frac{by-br}{aq_{2}}-z\biggl)\nonumber\\
&&\qquad\times\overline{B}_{m+n-j}\biggl(\frac{k+x}{cq_{1}q_{2}}+\frac{l}{a}+\frac{y-r}{aq_{2}}\biggl)\nonumber\\
&&\qquad=\sum_{t=0}^{cq_{2}-1}\sum_{s=1}^{q_{1}}\chi_{1}(q_{1}t+s)\sum_{r=1}^{q_{2}}\overline{\chi_{2}}(r)\sum_{l=1}^{a}\overline{B}_{j}\biggl(\frac{bl}{a}+\frac{by-br}{aq_{2}}-z\biggl)\nonumber\\
&&\qquad\qquad\times\overline{B}_{m+n-j}\biggl(\frac{t}{cq_{2}}+\frac{s+x}{cq_{1}q_{2}}+\frac{l}{a}+\frac{y-r}{aq_{2}}\biggl)\nonumber\\
&&\qquad=\frac{1}{(cq_{2})^{m+n-j-1}}\sum_{s=1}^{q_{1}}\chi_{1}(s)\sum_{r=1}^{q_{2}}\overline{\chi_{2}}(r)\sum_{l=1}^{a}\overline{B}_{j}\biggl(\frac{bl}{a}+\frac{by-br}{aq_{2}}-z\biggl)\nonumber\\
&&\qquad\qquad\times\overline{B}_{m+n-j}\biggl(\frac{s+x}{q_{1}}+\frac{cq_{2}l}{a}+\frac{c(y-r)}{a}\biggl)\nonumber\\
&&\qquad=\frac{\chi_{2}(-1)}{(cq_{1}q_{2})^{m+n-j-1}}\sum_{r=1}^{q_{2}}\sum_{l=1}^{a}\overline{\chi_{2}}(q_{2}l-r)\overline{B}_{j}\biggl(\frac{bl}{a}+\frac{by-br}{aq_{2}}-z\biggl)\nonumber\\
&&\qquad\qquad\times\overline{B}_{m+n-j}\biggl(x+\frac{cq_{1}q_{2}l}{a}+\frac{cq_{1}(y-r)}{a},\chi_{1}\biggl)\nonumber\\
&&\qquad=\frac{\chi_{2}(-1)}{(cq_{1}q_{2})^{m+n-j-1}}\sum_{k=1}^{aq_{2}}\overline{\chi_{2}}(k)\overline{B}_{j}\biggl(\frac{b(k+y)}{aq_{2}}-z\biggl)\nonumber\\
&&\qquad\qquad\times\overline{B}_{m+n-j}\biggl(\frac{cq_{1}(k+y)}{a}+x,\chi_{1}\biggl).
\end{eqnarray}
Therefor, inserting \eqref{eq3.24} and \eqref{eq3.25} into \eqref{eq3.23} and then replacing $a$ by $aq_{1}$, $x$ by $z$, $y$ by $x$, and $z$ by $y$, we obtain \eqref{eq1.19}. This concludes the proof of Theorem \ref{thm1.4}.

\section{Some special cases}\label{sec4}

We begin by discussing a special case of Theorem \ref{thm1.2}. We define the sums $s_{n}(a,b;\chi_{1},\overline{\chi_{2}};x,y)$ as follows:
\begin{equation}\label{eq4.1}
s_{n}(a,b;\chi_{1},\overline{\chi_{2}};x,y)=\sum_{r=1}^{bq}\chi_{1}(r)\overline{B}_{n}\biggl(\frac{a(r+y)}{b}+x,\overline{\chi_{2}}\biggl)\overline{B}_{1}\biggl(\frac{r+y}{bq}\biggl),
\end{equation}
where $n,q\in\mathbb{N}$, $a,b\in\mathbb{Z}$ with $b\geq1$, $\chi_{1}$ and $\chi_{2}$ are Dirichlet characters modulo $q$, and $x,y\in\mathbb{R}$.
The reciprocity formula for the sums \eqref{eq4.1} is given in the following corollary.

\begin{corollary}\label{cor4.1} Let $a,b,n,q\in\mathbb{N}$ with $q\geq 2$. Let $\chi_{1}$ and $\chi_{2}$ be non-principal primitive Dirichlet characters modulo $q$. Then, for $x,y\in\mathbb{R}$,
\begin{eqnarray}\label{eq4.2}
&&ab^{n}s_{n}(a,b;\chi_{1},\overline{\chi_{2}};x,y)+ba^{n}s_{n}(b,a;\overline{\chi_{2}},\chi_{1};y,x)\nonumber\\
&&\qquad=\frac{1}{n+1}\sum_{j=1}^{n}\binom{n+1}{j}a^{j}b^{n+1-j}\overline{B}_{j}(y,\chi_{1})\overline{B}_{n+1-j}(x,\overline{\chi_{2}})\nonumber\\
&&\qquad\quad+\frac{nq^{n-1}(a,b)^{n+1}}{n+1}\sum_{k=1}^{q}\sum_{j=1}^{q}\chi_{1}(k)\overline{\chi_{2}}(j)\overline{B}_{n+1}\biggl(\frac{a(k+y)+b(j+x)}{q(a,b)}\biggl)\nonumber\\
&&\qquad\quad-\frac{\delta_{1,n}\chi_{2}(-1)ab}{4}\sum_{k=1}^{bq}\chi_{1}(k)\delta_{\mathbb{Z}}\biggl(\frac{a(k+y)}{b}+x\biggl)\delta_{\mathbb{Z}}\biggl(\frac{k+y}{bq}\biggl)\nonumber\\
&&\qquad\qquad\times\overline{\chi_{2}}\biggl(\frac{a(k+y)}{b}+x\biggl).
\end{eqnarray}
\end{corollary}

\begin{proof}
Taking $n=b=1$ and $y=0$ in Theorem \ref{thm1.2} and then replacing $m$ by $n$, $c$ by $b$, and $z$ by $y$, we obtain
\begin{eqnarray}\label{eq4.3}
&&s_{n}(a,b;\chi_{1},\overline{\chi_{2}};x,y)+\biggl(\frac{a}{b}\biggl)^{n-1}s_{n}(b,a;\overline{\chi_{2}},\chi_{1};y,x)\nonumber\\
&&\qquad=\frac{1}{n+1}\sum_{j=1}^{n}\binom{n+1}{j}\biggl(\frac{a}{b}\biggl)^{n-j}\overline{B}_{j}(x,\overline{\chi_{2}})\overline{B}_{n+1-j}(y,\chi_{1})\nonumber\\
&&\qquad\quad+\frac{na^{n-1}}{b^{n}q(n+1)}\sum_{r=1}^{aq}\overline{\chi_{2}}(r)\overline{B}_{n+1}\biggl(\frac{b(r+x)}{a}+y,\chi_{1}\biggl)\nonumber\\
&&\qquad\quad-\frac{\delta_{1,n}\chi_{2}(-1)}{4}\sum_{k=1}^{bq}\chi_{1}(k)\delta_{\mathbb{Z}}\biggl(\frac{a(k+y)}{b}+x\biggl)\delta_{\mathbb{Z}}\biggl(\frac{k+y}{bq}\biggl)\nonumber\\
&&\qquad\qquad\times\overline{\chi_{2}}\biggl(\frac{a(k+y)}{b}+x\biggl).
\end{eqnarray}
From \eqref{eq2.14}, the property of residue systems, and the division algorithm, it is easily seen that, for $a,j\in\mathbb{N}$, $b\in\mathbb{Z}$ with $b\neq0$, and $x\in\mathbb{R}$,
\begin{eqnarray}\label{eq4.4}
\sum_{l=0}^{a-1}\overline{B}_{j}\biggl(\frac{x+bl}{a}\biggl)
&=&\sum_{r=0}^{(a,b)-1}\sum_{j=0}^{\frac{a}{(a,b)}-1}\overline{B}_{j}\biggl(\frac{\frac{x}{(a,b)}+\frac{b}{(a,b)}(\frac{a}{(a,b)}r+j)}{\frac{a}{(a,b)}}\biggl)\nonumber\\
&=&(a,b)\sum_{j=0}^{\frac{a}{(a,b)}-1}\overline{B}_{j}\biggl(\frac{\frac{x}{(a,b)}+\frac{bj}{(a,b)}}{\frac{a}{(a,b)}}\biggl)\nonumber\\
&=&(a,b)\sum_{j=0}^{\frac{a}{(a,b)}-1}\overline{B}_{j}\biggl(\frac{\frac{x}{(a,b)}+j}{\frac{a}{(a,b)}}\biggl)\nonumber\\
&=&(a,b)^{j}a^{1-j}\overline{B}_{j}\biggl(\frac{x}{(a,b)}\biggl).
\end{eqnarray}
It follows from \eqref{eq2.8} and \eqref{eq4.4} that
\begin{eqnarray*}
&&\sum_{r=1}^{aq}\overline{\chi_{2}}(r)\overline{B}_{n+1}\biggl(\frac{b(r+x)}{a}+y,\chi_{1}\biggl)\nonumber\\
&&\qquad=q^{n}\sum_{r=1}^{aq}\overline{\chi_{2}}(r)\sum_{k=1}^{q}\chi_{1}(k)\overline{B}_{n+1}\biggl(\frac{b(r+x)}{aq}+\frac{k+y}{q}\biggl)\nonumber\\
&&\qquad=q^{n}\sum_{l=0}^{a-1}\sum_{j=1}^{q}\overline{\chi_{2}}(ql+j)\sum_{k=1}^{q}\chi_{1}(k)\overline{B}_{n+1}\biggl(\frac{b(j+x)}{aq}+\frac{bl}{a}+\frac{k+y}{q}\biggl)\nonumber\\
&&\qquad=\frac{q^{n}(a,b)^{n+1}}{a^{n}}\sum_{j=1}^{q}\overline{\chi_{2}}(j)\sum_{k=1}^{q}\chi_{1}(k)\overline{B}_{n+1}\biggl(\frac{b(j+x)+a(k+y)}{q(a,b)}\biggl).
\end{eqnarray*}
This, together with \eqref{eq4.3}, yields
\begin{eqnarray}\label{eq4.5}
&&s_{n}(a,b;\chi_{1},\overline{\chi_{2}};x,y)+\biggl(\frac{a}{b}\biggl)^{n-1}s_{n}(b,a;\overline{\chi_{2}},\chi_{1};y,x)\nonumber\\
&&\qquad=\frac{1}{n+1}\sum_{j=1}^{n}\binom{n+1}{j}\biggl(\frac{a}{b}\biggl)^{n-j}\overline{B}_{j}(x,\overline{\chi_{2}})\overline{B}_{n+1-j}(y,\chi_{1})\nonumber\\
&&\qquad\quad+\frac{nq^{n-1}(a,b)^{n+1}}{ab^{n}(n+1)}\sum_{k=1}^{q}\sum_{j=1}^{q}\chi_{1}(k)\overline{\chi_{2}}(j)\overline{B}_{n+1}\biggl(\frac{a(k+y)+b(j+x)}{q(a,b)}\biggl)\nonumber\\
&&\qquad\quad-\frac{\delta_{1,n}\chi_{2}(-1)}{4}\sum_{k=1}^{bq}\chi_{1}(k)\delta_{\mathbb{Z}}\biggl(\frac{a(k+y)}{b}+x\biggl)\delta_{\mathbb{Z}}\biggl(\frac{k+y}{bq}\biggl)\nonumber\\
&&\qquad\qquad\times\overline{\chi_{2}}\biggl(\frac{a(k+y)}{b}+x\biggl).
\end{eqnarray}
Now, multiplying both sides of \eqref{eq4.5} by $ab^{n}$ gives the desired result.
\end{proof}

It is clear that, taking $n=1$, $\chi_{1}=\chi_{2}=\chi$, $(a,b)=1$, $a\equiv0$ or $b\equiv0\pmod{q}$, and $x=y=0$ in Corollary \ref{cor4.1}, we rediscover Berndt's reciprocity formula \eqref{eq1.4}. It is also easy to see that, taking $\chi_{1}=\chi_{2}=\chi$ and $(a,b)=1$ in Corollary \ref{cor4.1}, we obtain an equivalent version of the reciprocity formula of Cenkci, Can and Kurt \cite[Theorem 5.4]{cenkci}. In addition, Corollary \ref{cor4.1} also implies that, for $a,b,n,q\in\mathbb{N}$ with $q\geq 2$, and for non-principal primitive Dirichlet characters $\chi_{1}$ and $\chi_{2}$ modulo $q$ with $(-1)^{n+1}\chi_{1}(-1)\chi_{2}(-1)=1$,
\begin{eqnarray}\label{eq4.6}
&&ab^{n}s_{n}(a,b;\chi_{1},\overline{\chi_{2}};0,0)+ba^{n}s_{n}(b,a;\overline{\chi_{2}},\chi_{1};0,0)\nonumber\\
&&\qquad=\frac{1}{n+1}\sum_{j=1}^{n}\binom{n+1}{j}a^{j}b^{n+1-j}B_{j}(\chi_{1})B_{n+1-j}(\overline{\chi_{2}})\nonumber\\
&&\qquad\quad+\frac{nq^{n-1}(a,b)^{n+1}}{n+1}\sum_{k=1}^{q}\sum_{j=1}^{q}\chi_{1}(k)\overline{\chi_{2}}(j)\overline{B}_{n+1}\biggl(\frac{ak+bj}{q(a,b)}\biggl).
\end{eqnarray}
This means that Da\u{g}l{\i} and Can's reciprocity formula \cite[Theorem 1]{dagli} also holds in the case $n=1$.

Similarly, we define the sums $s_{n}^{(1)}(a,b;\chi;x,y)$ and $s_{n}^{(2)}(a,b;\chi;x,y)$ by
\begin{equation}\label{eq4.7}
s_{n}^{(1)}(a,b;\chi;x,y)=\sum_{r=1}^{bq}\chi(r)\overline{B}_{n}\biggl(\frac{a(r+y)}{b}+x\biggl)\overline{B}_{1}\biggl(\frac{r+y}{bq}\biggl),
\end{equation}
and
\begin{equation}\label{eq4.8}
s_{n}^{(2)}(a,b;\chi;x,y)=\sum_{r=1}^{bq}\overline{B}_{n}\biggl(\frac{a(r+y)}{b}+x,\chi\biggl)\overline{B}_{1}\biggl(\frac{r+y}{bq}\biggl),
\end{equation}
where $n,q\in\mathbb{N}$, $a,b\in\mathbb{Z}$ with $b\geq1$, $\chi$ is a Dirichlet character modulo $q$, and $x,y\in\mathbb{R}$.

As a consequence of Theorem \ref{thm1.3}, we have the reciprocity formula relating the sums \eqref{eq4.7} and \eqref{eq4.8} as follows.

\begin{corollary}\label{cor4.2} Let $a,b,n,q\in\mathbb{N}$ with $q\geq 2$. Let $\chi$ be a non-principal primitive Dirichlet character modulo $q$. Then, for $x,y\in\mathbb{R}$,
\begin{eqnarray}\label{eq4.9}
&&ab^{n}s_{n}^{(1)}(a,b;\chi;x,y)+ba^{n}s_{n}^{(2)}(b,a;\chi;y,x)\nonumber\\
&&\qquad=\frac{1}{n+1}\sum_{j=1}^{n+1}\binom{n+1}{j}a^{j}b^{n+1-j}\overline{B}_{j}(y,\chi)\overline{B}_{n+1-j}(x)\nonumber\\
&&\qquad\quad+\frac{n(b,aq)^{n+1}}{q(n+1)}\sum_{k=1}^{q}\chi(k)\overline{B}_{n+1}\biggl(\frac{a(k+y)+bx}{(b,aq)}\biggl)\nonumber\\
&&\qquad\quad-\frac{\delta_{1,n}\chi(-1)ab}{4}\sum_{k=1}^{aq}\delta_{\mathbb{Z}}\biggl(\frac{b(k+x)}{a}+y\biggl)\delta_{\mathbb{Z}}\biggl(\frac{k+x}{aq}\biggl)\nonumber\\
&&\qquad\qquad\times\chi\biggl(\frac{b(k+x)}{a}+y\biggl).
\end{eqnarray}
\end{corollary}

\begin{proof}
Taking $n=b=1$ and $y=0$ in Theorem \ref{thm1.3} and then replacing $m$ by $n$, $c$ by $b$, and $z$ by $y$, we find that
\begin{eqnarray}\label{eq4.10}
&&s_{n}^{(2)}(a,b;\chi;x,y)+\biggl(\frac{a}{b}\biggl)^{n-1}s_{n}^{(1)}(b,a;\chi;y,x)\nonumber\\
&&\qquad=\frac{1}{n+1}\sum_{j=1}^{n+1}\binom{n+1}{j}\biggl(\frac{a}{b}\biggl)^{n-j}\overline{B}_{j}(x,\chi)\overline{B}_{n+1-j}(y)\nonumber\\
&&\qquad\quad+\frac{na^{n-1}}{b^{n}q(n+1)}\sum_{r=1}^{aq}\chi(r)\overline{B}_{n+1}\biggl(\frac{b(r+x)}{a}+y\biggl)\nonumber\\
&&\qquad\quad-\frac{\delta_{1,n}\chi(-1)}{4}\sum_{k=1}^{bq}\delta_{\mathbb{Z}}\biggl(\frac{a(k+y)}{b}+x\biggl)\delta_{\mathbb{Z}}\biggl(\frac{k+y}{bq}\biggl)\nonumber\\
&&\qquad\qquad\times\chi\biggl(\frac{a(k+y)}{b}+x\biggl).
\end{eqnarray}
Note that from \eqref{eq4.4} and the division algorithm we have
\begin{eqnarray}\label{eq4.11}
&&\sum_{r=1}^{aq}\chi(r)\overline{B}_{n+1}\biggl(\frac{b(r+x)}{a}+y\biggl)\nonumber\\
&&\qquad=\sum_{l=0}^{a-1}\sum_{k=1}^{q}\chi(ql+k)\overline{B}_{n+1}\biggl(\frac{b(ql+k+x)}{a}+y\biggl)\nonumber\\
&&\qquad=\frac{(a,bq)^{n+1}}{a^{n}}\sum_{k=1}^{q}\chi(k)\overline{B}_{n+1}\biggl(\frac{b(k+x)+ay}{(a,bq)}\biggl).
\end{eqnarray}
Applying \eqref{eq4.11} to the right-hand side of \eqref{eq4.10}, we obtain
\begin{eqnarray*}
&&s_{n}^{(2)}(a,b;\chi;x,y)+\biggl(\frac{a}{b}\biggl)^{n-1}s_{n}^{(1)}(b,a;\chi;y,x)\nonumber\\
&&\qquad=\frac{1}{n+1}\sum_{j=1}^{n+1}\binom{n+1}{j}\biggl(\frac{a}{b}\biggl)^{n-j}\overline{B}_{j}(x,\chi)\overline{B}_{n+1-j}(y)\nonumber\\
&&\qquad\quad+\frac{n(a,bq)^{n+1}}{ab^{n}q(n+1)}\sum_{k=1}^{q}\chi(k)\overline{B}_{n+1}\biggl(\frac{b(k+x)+ay}{(a,bq)}\biggl)\nonumber\\
&&\qquad\quad-\frac{\delta_{1,n}\chi(-1)}{4}\sum_{k=1}^{bq}\delta_{\mathbb{Z}}\biggl(\frac{a(k+y)}{b}+x\biggl)\delta_{\mathbb{Z}}\biggl(\frac{k+y}{bq}\biggl)\chi\biggl(\frac{a(k+y)}{b}+x\biggl).
\end{eqnarray*}
Therefore, multiplying both sides of the above identity by $ab^{n}$ and then swapping $a$ with $b$ and $x$ with $y$, we obtain the desired result.
\end{proof}

It is obvious that, taking $x=y=0$ in Corollary \ref{cor4.2}, we obtain the following result: for $a,b,n,q\in\mathbb{N}$ with $q\geq 2$, and for a non-principal primitive Dirichlet character $\chi$ modulo $q$ with $(-1)^{n+1}\chi(-1)=1$,
\begin{eqnarray}\label{eq4.12}
&&ab^{n}s_{n}^{(1)}(a,b;\chi;0,0)+ba^{n}s_{n}^{(2)}(b,a;\chi;0,0)\nonumber\\
&&\qquad=\frac{1}{n+1}\sum_{j=1}^{n+1}\binom{n+1}{j}a^{j}b^{n+1-j}B_{j}(\chi)\overline{B}_{n+1-j}(0)\nonumber\\
&&\qquad\quad+\frac{n(b,aq)^{n+1}}{q(n+1)}\sum_{k=1}^{q}\chi(k)\overline{B}_{n+1}\biggl(\frac{ak}{(b,aq)}\biggl).
\end{eqnarray}
In particular, the case $n=1$ and $(a,b)=1$ in \eqref{eq4.12} gives Berndt's reciprocity formulas \eqref{eq1.7} and \eqref{eq1.8}.

As another consequence of Theorem \ref{thm1.3}, we have the following three-term reciprocity formula, which removes the hypotheses $(b,q)=1$ and $(a,b)=(b,c)=(c,a)=1$ in Berndt's reciprocity formulas \cite[Corollary 8.5]{berndt4}.

\begin{corollary}\label{cor4.3} Let $a,b,c,q\in\mathbb{N}$ with $q\geq 2$. Let $\chi$ be an even, non-principal primitive Dirichlet character modulo $q$. Then
\begin{eqnarray}\label{eq4.13}
&&s_{1,1}^{(1)}(c,b,a;\chi;0,0,0)+s_{1,1}^{(2)}(a,b,c;\chi;0,0,0)+s_{1,1}^{(3)}(a,c,b;\chi;0,0,0)\nonumber\\
&&\qquad=\frac{b(a,cq)^{2}}{2acq}\sum_{j=1}^{q}\chi(j)\overline{B}_{2}\biggl(\frac{cj}{(a,cq)}\biggl)+\frac{(a,b)^{2}c\overline{\chi}(\frac{b}{(a,b)})B_{2}(\chi)}{2ab}\nonumber\\
&&\qquad\quad-\frac{1}{4}\sum_{k=1}^{cq}\delta_{\mathbb{Z}}\biggl(\frac{ak}{c}\biggl)\delta_{\mathbb{Z}}\biggl(\frac{bk}{cq}\biggl)\chi\biggl(\frac{ak}{c}\biggl).
\end{eqnarray}
\end{corollary}

\begin{proof} Taking $m=n=1$ and $x=y=z=0$ in Theorem \ref{thm1.3} and using \eqref{eq2.7}, we have
\begin{eqnarray}\label{eq4.14}
&&s_{1,1}^{(1)}(c,b,a;\chi;0,0,0)+s_{1,1}^{(2)}(a,b,c;\chi;0,0,0)+s_{1,1}^{(3)}(a,c,b;\chi;0,0,0)\nonumber\\
&&\qquad=\frac{b}{2cq}\sum_{r=1}^{aq}\chi(r)\overline{B}_{2}\biggl(\frac{cr}{a}\biggl)+\frac{(a,b)^{2}c\overline{\chi}(\frac{b}{(a,b)})B_{2}(\chi)}{2ab}\nonumber\\
&&\qquad\quad-\frac{1}{4}\sum_{k=1}^{cq}\delta_{\mathbb{Z}}\biggl(\frac{ak}{c}\biggl)\delta_{\mathbb{Z}}\biggl(\frac{bk}{cq}\biggl)\chi\biggl(\frac{ak}{c}\biggl).
\end{eqnarray}
It is easily seen from \eqref{eq4.4} that
\begin{eqnarray}\label{eq4.15}
\sum_{r=1}^{aq}\chi(r)\overline{B}_{2}\biggl(\frac{cr}{a}\biggl)&=&\sum_{l=0}^{a-1}\sum_{j=1}^{q}\chi(ql+j)\overline{B}_{2}\biggl(\frac{c(ql+j)}{a}\biggl)\nonumber\\
&=&\frac{(a,cq)^{2}}{a}\sum_{j=1}^{q}\chi(j)\overline{B}_{2}\biggl(\frac{cj}{(a,cq)}\biggl).
\end{eqnarray}
Thus, the desired result follows immediately by applying \eqref{eq4.15} to the right-hand side of \eqref{eq4.14}.
\end{proof}

We next consider a special case of Theorem \ref{thm1.4}. Like the sums \eqref{eq4.1}, the sums $\widehat{s}_{n}(a,b;\chi_{1},\chi_{2};x,y)$ are defined by
\begin{equation}\label{eq4.16}
\widehat{s}_{n}(a,b;\chi_{1},\chi_{2};x,y)=\sum_{r=1}^{cq_{1}q_{2}}\chi_{1}(r)\overline{B}_{n}\biggl(\frac{a(r+y)}{c}+x,\chi_{2}\biggl)\overline{B}_{1}\biggl(\frac{r+y}{cq_{1}q_{2}}\biggl),
\end{equation}
where $n,q_{1},q_{2}\in\mathbb{N}$, $a,b\in\mathbb{Z}$ with $b\geq1$, $\chi_{1}$ and $\chi_{2}$ are Dirichlet characters modulo $q_{1}$ and $q_{2}$, respectively, and $x,y\in\mathbb{R}$. The following corollary gives the reciprocity formula for the sums \eqref{eq4.16}.

\begin{corollary}\label{cor4.4} Let $a,b,n,q_{1},q_{2}\in\mathbb{N}$ with $q_{1}\geq2$ and $q_{2}\geq2$. Let $\chi_{1}$ and $\chi_{2}$ be Dirichlet characters modulo $q_{1}$ and $q_{2}$, respectively. Then, for $x,y\in\mathbb{R}$,
\begin{eqnarray}\label{eq4.17}
&&ab^{n}\widehat{s}_{n}(a,b;\chi_{1},\overline{\chi_{2}};x,y)+ba^{n}\widehat{s}_{n}(b,a;\overline{\chi_{2}},\chi_{1};y,x)\nonumber\\
&&\qquad=\frac{1}{n+1}\sum_{j=1}^{n}\binom{n+1}{j}a^{j}b^{n+1-j}\overline{B}_{j}(y,\chi_{1})\overline{B}_{n+1-j}(x,\overline{\chi_{2}})\nonumber\\
&&\qquad\quad+\frac{n(aq_{1},bq_{2})^{n+1}}{q_{1}q_{2}(n+1)}\sum_{k=1}^{q_{1}}\sum_{j=1}^{q_{2}}\chi_{1}(k)\overline{\chi_{2}}(j)\overline{B}_{n+1}\biggl(\frac{a(k+y)+b(j+x)}{(aq_{1},bq_{2})}\biggl)\nonumber\\
&&\qquad\quad-\frac{\delta_{1,n}\chi_{2}(-1)ab}{4}\sum_{k=1}^{bq_{1}q_{2}}\chi_{1}(k)\delta_{\mathbb{Z}}\biggl(\frac{a(k+y)}{b}+x\biggl)
\delta_{\mathbb{Z}}\biggl(\frac{k+y}{bq_{1}q_{2}}\biggl)\nonumber\\
&&\qquad\qquad\times\overline{\chi_{2}}\biggl(\frac{a(k+y)}{b}+x\biggl).
\end{eqnarray}
\end{corollary}

\begin{proof} Taking $n=b=1$ and $y=0$ in Theorem \ref{thm1.4} and then replacing $m$ by $n$, $c$ by $b$, and $z$ by $y$, we know that
\begin{eqnarray}\label{eq4.18}
&&\widehat{s}_{n}(a,b;\chi_{1},\overline{\chi_{2}};x,y)+\biggl(\frac{a}{b}\biggl)^{n-1}\widehat{s}_{n}(b,a;\overline{\chi_{2}},\chi_{1};y,x)\nonumber\\
&&\qquad=\frac{1}{n+1}\sum_{j=1}^{n}\binom{n+1}{j}\biggl(\frac{a}{b}\biggl)^{n-j}\overline{B}_{j}(x,\overline{\chi_{2}})\overline{B}_{n+1-j}(y,\chi_{1})\nonumber\\
&&\qquad\quad+\frac{na^{n-1}}{b^{n}q_{1}q_{2}(n+1)}\sum_{r=1}^{aq_{1}q_{2}}\overline{\chi_{2}}(r)\overline{B}_{n+1}\biggl(\frac{b(r+x)}{a}+y,\chi_{1}\biggl)\nonumber\\
&&\qquad\quad-\frac{\delta_{1,n}\chi_{2}(-1)}{4}\sum_{k=1}^{bq_{1}q_{2}}\chi_{1}(k)\delta_{\mathbb{Z}}\biggl(\frac{a(k+y)}{b}+x\biggl)
\delta_{\mathbb{Z}}\biggl(\frac{k+y}{bq_{1}q_{2}}\biggl)\nonumber\\
&&\qquad\qquad\times\overline{\chi_{2}}\biggl(\frac{a(k+y)}{b}+x\biggl).
\end{eqnarray}
From \eqref{eq2.8} and \eqref{eq4.4}, we obtain
\begin{eqnarray*}
&&\sum_{r=1}^{aq_{1}q_{2}}\overline{\chi_{2}}(r)\overline{B}_{n+1}\biggl(\frac{b(r+x)}{a}+y,\chi_{1}\biggl)\nonumber\\
&&\qquad=q_{1}^{n}\sum_{r=1}^{aq_{1}q_{2}}\overline{\chi_{2}}(r)\sum_{k=1}^{q_{1}}\chi_{1}(k)\overline{B}_{n+1}\biggl(\frac{b(r+x)}{aq_{1}}+\frac{k+y}{q_{1}}\biggl)\nonumber\\
&&\qquad=q_{1}^{n}\sum_{l=0}^{aq_{1}-1}\sum_{j=1}^{q_{2}}\overline{\chi_{2}}(q_{2}l+j)\sum_{k=1}^{q_{1}}\chi_{1}(k)\overline{B}_{n+1}\biggl(\frac{b(j+x)}{aq_{1}}+\frac{bq_{2}l}{aq_{1}}+\frac{k+y}{q_{1}}\biggl)\nonumber\\
&&\qquad=\frac{(aq_{1},bq_{2})^{n+1}}{a^{n}}\sum_{j=1}^{q_{2}}\overline{\chi_{2}}(j)\sum_{k=1}^{q_{1}}\chi_{1}(k)\overline{B}_{n+1}\biggl(\frac{b(j+x)+a(k+y)}{(aq_{1},bq_{2})}\biggl).
\end{eqnarray*}
This, together with \eqref{eq4.18}, yields
\begin{eqnarray}\label{eq4.19}
&&\widehat{s}_{n}(a,b;\chi_{1},\overline{\chi_{2}};x,y)+\biggl(\frac{a}{b}\biggl)^{n-1}\widehat{s}_{n}(b,a;\overline{\chi_{2}},\chi_{1};y,x)\nonumber\\
&&\qquad=\frac{1}{n+1}\sum_{j=1}^{n}\binom{n+1}{j}\biggl(\frac{a}{b}\biggl)^{n-j}\overline{B}_{j}(x,\overline{\chi_{2}})\overline{B}_{n+1-j}(y,\chi_{1})\nonumber\\
&&\qquad\quad+\frac{n(aq_{1},bq_{2})^{n+1}}{ab^{n}q_{1}q_{2}(n+1)}\sum_{k=1}^{q_{1}}\sum_{j=1}^{q_{2}}\chi_{1}(k)\overline{\chi_{2}}(j)\overline{B}_{n+1}\biggl(\frac{a(k+y)+b(j+x)}{(aq_{1},bq_{2})}\biggl)\nonumber\\
&&\qquad\quad-\frac{\delta_{1,n}\chi_{2}(-1)}{4}\sum_{k=1}^{bq_{1}q_{2}}\chi_{1}(k)\delta_{\mathbb{Z}}\biggl(\frac{a(k+y)}{b}+x\biggl)
\delta_{\mathbb{Z}}\biggl(\frac{k+y}{bq_{1}q_{2}}\biggl)\nonumber\\
&&\qquad\qquad\times\overline{\chi_{2}}\biggl(\frac{a(k+y)}{b}+x\biggl).
\end{eqnarray}
Thus, multiplying both sides of \eqref{eq4.19} by $ab^{n}$, we obtain the desired result.
\end{proof}

In particular, taking $x=y=0$ in Corollary \ref{cor4.4}, we see that for $a,b,n,q_{1},q_{2}\in\mathbb{N}$ with $q_{1}\geq2$, $q_{2}\geq2$, and $q_{1}\neq q_{2}$, and for non-principal primitive Dirichlet characters $\chi_{1}$ and $\chi_{2}$ modulo $q_{1}$ and $q_{2}$, respectively, if $(-1)^{n+1}\chi_{1}(-1)\chi_{2}(-1)=1$, then
\begin{eqnarray}\label{eq4.20}
&&ab^{n}\widehat{s}_{n}(a,b;\chi_{1},\overline{\chi_{2}};0,0)+ba^{n}\widehat{s}_{n}(b,a;\overline{\chi_{2}},\chi_{1};0,0)\nonumber\\
&&\qquad=\frac{1}{n+1}\sum_{j=1}^{n}\binom{n+1}{j}a^{j}b^{n+1-j}B_{j}(\chi_{1})B_{n+1-j}(\overline{\chi_{2}})\nonumber\\
&&\qquad\quad+\frac{n(aq_{1},bq_{2})^{n+1}}{q_{1}q_{2}(n+1)}\sum_{k=1}^{q_{1}}\sum_{j=1}^{q_{2}}\chi_{1}(k)\overline{\chi_{2}}(j)\overline{B}_{n+1}\biggl(\frac{ak+bj}{(aq_{1},bq_{2})}\biggl).
\end{eqnarray}
This indicates an error in the reciprocity formula of Da\u{g}l{\i} and Can \cite[Corollary 3]{dagli}. Specifically, their formula omits the last term appearing on the right-hand side of \eqref{eq4.20}.

\end{document}